\documentclass[11pt, english, twoside]{article}
\usepackage[margin=3cm]{geometry}
\usepackage{amsbsy,amsmath,amsthm,amssymb}

\usepackage{fourier}
\DeclareMathAlphabet{\mathcal}{OMS}{cmsy}{m}{n}

\usepackage[dvipsnames]{xcolor}
\usepackage{bm}
\usepackage{esint}
\usepackage{babel}
\usepackage[colorlinks=true, citecolor = blue]{hyperref}

\usepackage{indentfirst}
\usepackage{changepage, cases}	
\usepackage{setspace}
\usepackage{indentfirst}
\usepackage{graphicx}
\usepackage{calc}
\usepackage{orcidlink}
\usepackage{booktabs,multirow}
\usepackage{makecell}

\usepackage{fancyhdr}
\allowdisplaybreaks

\theoremstyle{plain}
\newtheorem{theorem}{Theorem}[section]
\newtheorem{lemma}[theorem]{Lemma}
\newtheorem{corollary}[theorem]{Corollary}
\newtheorem{proposition}[theorem]{Proposition}

\theoremstyle{remark}
\newtheorem{remark}[theorem]{Remark}

\theoremstyle{definition}

\newenvironment{proof of theorem 1.1}{{\noindent \em Proof of Theorem 1.1.}}{\hfill $\Box$\par}
\newenvironment{proof of theorem 1.2}{{\noindent \em Proof of Theorem 1.2.}}{\hfill $\Box$\par}

\DeclareSymbolFont{EulerExtension}{U}{euex}{m}{n}
\DeclareMathSymbol{\euintop}{\mathop} {EulerExtension}{"52}
\DeclareMathSymbol{\euointop}{\mathop} {EulerExtension}{"48}

\begin{document}
	\title{Multivariate Laguerre polynomials: new results and insights}
	\author{
		Liang-Jia Guo \orcidlink{0000-0002-1146-5896} $^{\rm 1}$,
		Min-Jie Luo \orcidlink{0000-0001-7433-4490} $^{\rm 1}$\thanks{Corresponding author.\vspace{3mm}\\ E-mail addresses: 
			\url{guoliangjia77@outlook.com} (L.-J. Guo); 
			\url{mathwinnie@live.com} (M.-J. Luo); \\
			\url{rkraina_7@hotmail.com} (R.K. Raina); \url{j1380533g@outlook.com} (J.-J. Wang)}, 
		Ravinder Krishna Raina $^{\rm 2}$, 
		Jia-Jun Wang $^{\rm 1}$}
	\date{}
	\maketitle
	\begin{center}\small
		$^{1}$\emph{School of Mathematics and Statistics, Donghua University,\\ 
			Shanghai 201620, People's Republic of China.}
	\end{center}
	\begin{center}\small
		$^{3}${\emph{M.P. University of Agriculture and Technology, Udaipur (Rajasthan), India\\
				\emph{Present address:} 10/11, Ganpati Vihar, Opposite Sector 5,\\
				Udaipur-313002, Rajasthan, India.}} 
	\end{center}
	
	
	\begin{abstract}
		In this paper, we study various properties of Erd\'{e}lyi's multivariate Laguerre polynomials $L_{n_1,\cdots,n_k}^{(\alpha)}(x_1,\cdots,x_k)$ including their generating functions, product formulas and fractional integral representations. Many useful consequences are derived. New insights into various classical results including the relationships of the multivariate Laguerre polynomials with Oshima’s fractional calculus operator and with an integral formula of Srivastava and Niukkanen are also mentioned. We further present an interesting evaluation for a generating function of the main diagonal sequence $L_{n,\cdots,n}^{(-\beta-kn)}(x_1,\cdots,x_k)$ which involves in a natural way the well-known Le Roy function ([Darboux Bull. 24 (2) (1899), 245--268]; [Toulouse Ann. 2 (2) (1900), 317--430]). The significance of the multivariate Laguerre polynomials $L_{n_1,\cdots,n_k}^{(\alpha)}(x_1,\cdots,x_k)$ is demonstrated by observing that this class not only includes the generalized Hardy-Hille formula and the product formula but also contains the multiple Laguerre polynomials of the second kind as its important special cases. We briefly indicate also possible lines of future work.  \\
		
		\\
		
	\end{abstract}

	
	\section{Introduction and objectives}\label{Introduction}

	The usual (unvariate) Laguerre polynomials $L_n^{(\alpha)}(x)$ can be defined by the generating function (\cite[p. 449, Eq. (18.12.13)]{NIST Handbook})
	\[
	(1-z)^{-\alpha-1}\exp\left(-\frac{xz}{1-z}\right)
	=\sum_{n=0}^{\infty}L_n^{(\alpha)}(x)z^n, ~ |z|<1.
	\]
	Explicitly, we have
	\begin{equation}\label{Def-LaguerreP}
		L_n^{(\alpha)}(z)=\frac{(\alpha+1)_n}{n!} {}_{1}F_{1}\left[\begin{matrix}
			-n\\
			\alpha+1
		\end{matrix};z\right],
	\end{equation}
where ${}_{1}F_{1}$ is the confluent hypergeometric function \cite[p. 443, Eq. (18.5.12)]{NIST Handbook}. During the past century, numerous results (e.g., generating functions, expansions, inequalities, and  asymptotics) have been obtained for the Laguerre polynomials $L_n^{(\alpha)}(z)$ and their variants (see, for example, \cite{Derezinski-Gas-Vatto-2025}, \cite{Frenzen-Wong-1988}, \cite{Koornwinder-1977}, \cite{Lewandowski-Szynal-1998}, \cite{Niukkanen-1984}, \cite{Ozmen-2019}, \cite{SanchezRuiz-LopezArtes-Dehesa-2003} and \cite{Srivastava-1989}). These results have not only motivated new developments in several areas of mathematics, but have also been extensively applied in physics. However, the multivariate cases remain largely unexplored. In this paper, we pay attention to a certain class of multivariate Laguerre polynomials.

There are many multivariate generalizations of the Laguerre polynomials $L_n^{(\alpha)}(x)$ in the literature (see, for example, \cite[p. 532, Eq. (1.4)]{Aktas-Erkus-Duman-2013}, \cite[p. 14, Eq. (1.1)]{Chak-1970}, \cite[p. 324, Definition 1]{Demenin-1971} and \cite[p. 387, Eq. (30)]{Liu-Lin-Lu-Srivastava-2013a}). There of course exists other types of higher-dimensional generalizations such as Chikuse's matrix-variate Laguerre polynomials \cite{Chikuse-1992}. But in the present work, we focus on the multivariate Laguerre polynomials introduced by Erd\'{e}lyi \cite{Erdelyi-1937} in 1937.  Erd\'{e}lyi's multivariate Laguerre polynomials of $k$-variables $L_{n_1,\cdots,n_k}^{(\alpha)}(x_1,\cdots,x_k)$ can be defined by the generating function
\begin{equation}\label{GF-MLaguerre}
	(1-z_1-\cdots-z_k)^{-\alpha-1}
	\exp\left\{-\frac{x_1 z_1+\cdots+x_k z_k}{1-z_1-\cdots-z_k}\right\}
	=\sum_{n_1,\cdots,n_k=0}^{\infty}L_{n_1,\cdots,n_k}^{(\alpha)}(x_1,\cdots,x_k)
	z_1^{n_1}\cdots z_k^{n_k},
\end{equation}
where $|z_1|+\cdots+|z_k|<1$, and it may be noted here that its particular case for $k=2$ was studied independently by Feldheim \cite{Feldheim-1943}.

For convenience, we first state the following notations which are used in the sequel. Boldface letters always denote vectors of dimension $k$; for instance, $\mathbf{n}=(n_1,\cdots,n_k)\in\mathbb{Z}_{\geq0}^k$ and $\mathbf{x}=(x_1,\cdots,x_k)\in\mathbb{C}^k$. For any $\mathbf{x},\mathbf{y}\in\mathbb{C}^k$, we let $\langle \mathbf{x}\rangle:=x_1+\cdots+x_k$, $\mathbf{x}\circ\mathbf{y}:=(x_1 y_1,\cdots, x_k y_k)$, $\mathbf{x}/\mathbf{y}:=(x_1/y_1,\cdots, x_k/y_k)$ and
$\mathbf{x}^{\mathbf{n}}:=x_1^{n_1}\cdots x_k^{n_k}$. For any  $\mathbf{n}\in\mathbb{Z}_{\geq0}^k$, we write $\mathbf{n}!=n_1!\cdots n_k!$. In addition, 
\[
\int_{[a,b]^k}f(\mathbf{x})\mathrm{d}\mathbf{x}~~~\text{means}~~~
\int_a^b\cdots\int_a^b f(x_1,\cdots,x_k)\mathrm{d}x_1\cdots\mathrm{d}x_k, 
\]
\[
\sum_{\mathbf{n}=\mathbf{0}}^{\infty}\Omega(\mathbf{n})~~~\text{means}~~~
\sum_{n_1,\cdots,n_k=0}^{\infty}\Omega(n_1,\cdots,n_k),
\]
and
\[
\sum_{\mathbf{j}=\mathbf{0}}^{\mathbf{n}}\Omega(\mathbf{j})~~~\text{means}~~~
\sum_{j_1=0}^{n_1}\cdots\sum_{j_k=0}^{n_k}\Omega(j_1,\cdots,j_k).
\]

Erd\'{e}lyi \cite[p. 458, Eq. (11.3)]{Erdelyi-1937} has shown that 
\begin{equation}\label{Def-MLaguerre}
	L_{\mathbf{n}}^{(\alpha)}(\mathbf{x})
	=\frac{(\alpha+1)_{\langle\mathbf{n}\rangle}}{\mathbf{n}!}\Phi_2^{(k)}\left[-n_1,\cdots,-n_k;\alpha+1;\mathbf{x}\right],
\end{equation}
where $\Phi_2^{(k)}$ is defined by (see \cite[p. 446]{Erdelyi-1937} and \cite[p. 34]{Srivastava-Karlsson-Book-1985})
\begin{equation}\label{ConfluentLauricella}
	\Phi_2^{(k)}\left[b_1,\cdots,b_k;c;\mathbf{x}\right]
	:=\sum_{\mathbf{m}=0}^{\infty}
	\frac{(b_1)_{m_1}\cdots (b_k)_{m_k}}{(c)_{\langle\mathbf{m}\rangle}}\frac{\mathbf{x}^{\mathbf{m}}}{\mathbf{m}!}.
\end{equation}
The function $\Phi_2^{(k)}$ can be considered as a confluent form of the Lauricella function $F_D^{(k)}$ or $F_B^{(k)}$ \cite[p. 34, Eq. (10)]{Srivastava-Karlsson-Book-1985}.

A very remarkable result concerning $L_{\mathbf{n}}^{(\alpha)}(\mathbf{x})$ that Erd\'{e}lyi has proved is the following elegant generalization of \emph{Hardy-Hille formula} \cite[p. 464, Eq. (14.5)]{Erdelyi-1937}:
\begin{equation}\label{GF-Erdelyi}
	\sum_{\mathbf{n}=\mathbf{0}}^{\infty}
	\frac{\mathbf{n}!L_{\mathbf{n}}^{(\alpha)}(\mathbf{x})L_{\mathbf{n}}^{(\alpha)}(\mathbf{y})}{\Gamma(\alpha+1+\langle\mathbf{n}\rangle)}\mathbf{u}^{\mathbf{n}}=\frac{(\mathbf{u};\mathbf{x},\mathbf{y})^{-\frac{\alpha}{2}}}{(1-\langle\mathbf{u}\rangle)^{\alpha+1}}
	\exp\left(-\frac{\langle\mathbf{u}\circ\mathbf{x}\rangle+\langle\mathbf{u}\circ\mathbf{y}\rangle}{1-\langle\mathbf{u}\rangle}\right)
	I_{\alpha}\big(2\sqrt{(\mathbf{u};\mathbf{x},\mathbf{y})}\big),
\end{equation}
where $I_\alpha(z)$ denotes the usual modified Bessel function \cite[p. 249, Eq. (10.25.2)]{NIST Handbook} and 
\[
(\mathbf{u};\mathbf{x},\mathbf{y}):=\frac{\langle\mathbf{u}\circ\mathbf{x}\circ\mathbf{y}\rangle-\langle\mathbf{u}\rangle\langle\mathbf{u}\circ\mathbf{x}\circ\mathbf{y}\rangle+\langle\mathbf{u}\circ\mathbf{x}\rangle\langle\mathbf{u}\circ\mathbf{y}\rangle}{(1-\langle\mathbf{u}\rangle)^{2}}.
\]
The formula \eqref{GF-Erdelyi}, in our opinion, makes this set of multivariate polynomials particularly useful and important. In a later work on the subject, Carlitz \cite{Carlitz-1970} provided an elementary proof of \eqref{GF-Erdelyi} and considered its other new generalizations.

Another important result deserving a mention here can be derived from the work of Liu \emph{et al.} \cite{Liu-Lin-Lu-Srivastava-2013b}. In their paper, integral representations for the product of two \emph{Carlitz-Srivastava polynomials of the first} and \emph{second kinds} are established. Due to the similarity of the Carlitz-Srivastava polynomials of the second kind with the multivariate Laguerre polynomials $L_{\mathbf{n}}^{(\alpha)}(\mathbf{x})$, we can easily obtain: 
\begin{align}\label{ProductFormula-MLaguerre}
	\frac{L_{\mathbf{m}}^{(\alpha)}(\mathbf{x})}{\Gamma(\alpha+1+\langle \mathbf{m}\rangle)}
	\frac{L_{\mathbf{n}}^{(\beta)}(\mathbf{y})}{\Gamma(\beta+1+\langle \mathbf{n}\rangle)}
	&=\frac{2^{\alpha+\beta+\langle \mathbf{m}+ \mathbf{n}\rangle}}{\pi^{k+1}}\int_{\left[-\frac{\pi}{2},\frac{\pi}{2}\right]^{k+1}}
	\mathrm{e}^{\mathrm{i}(\alpha-\beta)\theta+\mathrm{i}\langle(\mathbf{m}-\mathbf{n})\circ\bm{\varphi}\rangle} \cos^{\alpha+\beta}\theta \notag\\
	&\cdot\cos^{m_1+n_1} \varphi_1 \cdots \cos^{m_k+n_k} \varphi_k \frac{L_{\mathbf{m}+\mathbf{n}}^{(\alpha+\beta)}\left(\Xi(\mathbf{x},\mathbf{y}; \theta,\bm{\varphi})\right)}{\Gamma(\alpha+\beta+1+\langle \mathbf{m}+ \mathbf{n}\rangle)}  \mathrm{d}\theta \mathrm{d}\bm{\varphi}, 
\end{align}
where $\Xi(\mathbf{x},\mathbf{y}; \theta,\bm{\varphi})
=(\xi_1+\eta_1,\cdots, \xi_k+\eta_k)$ with 
\begin{align}
	\xi_j
	&\equiv\xi_j(x_j; \theta, \varphi_j):=x_j 
	\mathrm{e}^{\mathrm{i}(\theta-\varphi_j)}
	\sec\varphi_j 
	\cos \theta,\label{ProductFormula-MLaguerre-A}\\
	\eta_j
	&\equiv\eta_j(y_j; \theta, \varphi_j):=y_j\mathrm{e}^{\mathrm{i}(\varphi_j-\theta)}
	\sec\varphi_j 
	\cos \theta.\label{ProductFormula-MLaguerre-B}
\end{align}
When $k=1$, \eqref{ProductFormula-MLaguerre} reduces to Carlitz's result \cite[p. 26, Eq. (3)]{Carlitz-1962}. For more detailed information about Carlitz-Srivastava polynomials, the interested reader may refer to \cite{Carlitz-Srivastava-1976a} and \cite{Carlitz-Srivastava-1976b}. In Section \ref{Sec-ProductFormulas}, we will attempt further investigation on the result \eqref{ProductFormula-MLaguerre}.

In addition to the two results presented above, we demonstrate several other interesting studies related to $L_{\mathbf{n}}^{(\alpha)}(\mathbf{x})$. Exton \cite{Exton-1991} gave more generating functions for $L_{\mathbf{n}}^{(\alpha)}(\mathbf{x})$. \"{O}zarslan's work \cite{Ozarslan-2014} reveals a connection between $L_{\mathbf{n}}^{(\alpha)}(\mathbf{x})$ and singular integral equations. Recently, Luo and Raina \cite{Luo-Raina-2026} established two inequalities for $L_{\mathbf{n}}^{(\alpha)}(\mathbf{x})$, thereby generalizing the celebrated Szeg\H{o}'s inequality for $L_n^{(\alpha)}(x)$. In addition, as an application in mathematical statistics, multivariate Laguerre polynomials were used by Ong and Ng \cite{Ong-Ng-2013} to construct a bivariate noncentral negative binomial distribution.

Finally, we further highlight the importance of multivariate Laguerre polynomials $L_{\mathbf{n}}^{(\alpha)}(\mathbf{x})$ by introducing their connection to \emph{multiple Laguerre polynomials of the second kind}, which is usually defined by the orthogonality conditions (see, for example, \cite[p. 856]{Lee-2007})
\[
\int_{0}^{\infty}x^\ell  L_{\mathbf{n}}^{(\alpha;\bm{\beta})}(x)x^{\alpha}\mathrm{e}^{\beta_j x}\mathrm{d}x=0, \quad \ell=0,1,\cdots,n_j-1,
\]
for $j=1,\cdots,k$, where $\bm{\beta}=(\beta_1,\cdots,\beta_k)$ and $\alpha>-1$, $\beta_j<0$ and $\beta_i\neq\beta_j$; whenever $i\neq j$. Recently, there has been increasing interest in this class of polynomials (see, for example, \cite{Bleher-Kuijlaars-2005}, \cite[Chapter 23]{Ismail-Book-2005}, \cite{Lee-2007,Lee-2013} and \cite{Zhang-Filipuk-2014}). The multiple Laguerre polynomials of the second kind $L_{\mathbf{n}}^{(\alpha;\bm{\beta})}(x)$ can also be defined by their generating function as follows (see \cite[p. 858, Theorem 2.3]{Lee-2007} and \cite[p. 1068]{Lee-2013}): 
\begin{equation}\label{GF-multipleLaguerreP-1}
	(1-\langle\mathbf{t}\rangle)^{-\alpha-1}\exp\left\{\frac{\langle\bm{\beta}\circ\mathbf{t}\rangle x}{1-\langle\mathbf{t}\rangle}\right\}
	=\sum_{\mathbf{n}=\mathbf{0}}^{\infty} L_{\mathbf{n}}^{(\alpha;\bm{\beta})}(x)\frac{\mathbf{t}^{\mathbf{n}}}{\mathbf{n}!}.
\end{equation}
On the other hand, if we let $\mathbf{z}=\mathbf{t}$ and $\mathbf{x}=-x\bm{\beta}$ in \eqref{GF-MLaguerre}, we get
\begin{equation}\label{GF-multipleLaguerreP-2}
	(1-\langle\mathbf{t}\rangle)^{-\alpha-1}
	\exp\left\{\frac{\langle\bm{\beta}\circ\mathbf{t}\rangle x}{1-\langle\mathbf{t}\rangle}\right\}
	=\sum_{\mathbf{n}=\mathbf{0}}^{\infty}L_{\mathbf{n}}^{(\alpha)}(-\beta_1x,\cdots,-\beta_k x)\mathbf{t}^{\mathbf{n}}.
\end{equation}
Then, by comparing the coefficients of \eqref{GF-multipleLaguerreP-1} and \eqref{GF-multipleLaguerreP-2}, we obtain the following interesting relationship between the multivariate Laguerre polynomials defined by \eqref{Def-MLaguerre} with the multiple Laguerre polynomials of second kind \eqref{GF-multipleLaguerreP-1}:
\begin{equation}\label{Rel-MLaguerre-MultipleLaguerre}
	L_{\mathbf{n}}^{(\alpha;\bm{\beta})}(x)
	=\mathbf{n}! L_{\mathbf{n}}^{(\alpha)}(-\beta_1x,\cdots,-\beta_k x).
\end{equation}
From \eqref{Def-MLaguerre} and \eqref{Rel-MLaguerre-MultipleLaguerre}, some common explicit representations \cite[p. 209--210]{MartinezFinkelshtein-Morales-Perales-2026} for the multiple Laguerre polynomials of the second kind can be easily obtained. To the best of our knowledge, the formula \eqref{Rel-MLaguerre-MultipleLaguerre}, connecting the polynomials $L_{\mathbf{n}}^{(\alpha;\bm{\beta})}(x)$ and $L_{\mathbf{n}}^{(\alpha)}(\mathbf{x})$, has not been explicitly reported in the earlier literature. With the help of \eqref{Rel-MLaguerre-MultipleLaguerre}, all the results obtained in this paper can be related to the multiple Laguerre polynomials of the second kind through this expression.

We organize and pursue this paper as follows: In Sections \ref{Sec-GF-Type-I} and \ref{Sec-GF-Type-II}, we prove two classes of generating functions for $L_{\mathbf{n}}^{-\beta-\langle\mathbf{n}\rangle}(\mathbf{x})$ and explore their various applications. In Section \ref{Sec-ProductFormulas}, we present a multiple contour integral version of the product formula \eqref{ProductFormula-MLaguerre}. In Section \ref{Sec-OshimaFC}, we discuss the relationship between an integral formula of Srivastava and Niukkanen and Oshima's fractional calculus. Some concluding remarks are given in Section \ref{Sec-ConcludingRemarks}, which includes an interesting evaluation of a generating function of the main diagonal sequence $L_{n,\cdots,n}^{(-\beta-kn)}(x_1,\cdots,x_k)$ involving the well-known Le Roy function \cite{Le Roy-1899}.

\section{Generating functions of type I}\label{Sec-GF-Type-I}

\subsection{Main results}

For the (univariate) Laguerre polynomials $L_n^{(\alpha)}(x)$, the following generating function is well-known (see \cite[p. 57]{AbdulHalim-AlSalam-1963}; see also \cite[p. 1222]{BenCheikh-Lamiri-2007}):
\begin{equation}\label{GF-AbdulHalim-AlSalam} 
	\Phi_1\left[\alpha,\beta;\gamma;-u,-uw\right]
	=\sum_{n=0}^{\infty}\frac{(\alpha)_n}{(\gamma)_n} 
	L_{n}^{(-\beta-n)}(w) u^n, 
\end{equation}
where $\Phi_1$ denotes Humbert's bivariate confluent hypergeometric function defined by (\cite[p. 25, Eq. (16)]{Srivastava-Karlsson-Book-1985})
\[
\Phi_1[a,b;c;x,y]:=\sum_{m,n=0}^{\infty}\frac{(a)_{m+n}(b)_m}{(c)_{m+n}}\frac{x^m}{m!}\frac{y^n}{n!},\quad |x|<1,|y|<\infty.
\]
On letting $\alpha=\gamma$ in \eqref{GF-AbdulHalim-AlSalam}, we obtain the familiar result \cite[p. 4, Eq. (9)]{MCBride-Book-1971}  given by
\begin{equation}\label{GF-Laguerre-2}
	\mathrm{e}^{-uw}(1+u)^{-\beta}
	=\sum_{n=0}^{\infty}L_{n}^{(-\beta-n)}(w)u^n.
\end{equation}
In the literature, polynomials defined by $f_n^{\beta}(z):=(-1)^n L_n^{(-\beta-n)}(z)$ are also termed as the modified Laguerre polynomials.

The following proposition provides a multivariate generalization of \eqref{GF-AbdulHalim-AlSalam}.

\begin{proposition}\label{Prop-1}
	Let $\alpha\in\mathbb{C}$ and $-\beta,\gamma\in\mathbb{C}\setminus\mathbb{Z}_{\leq 0}$. 
	Also, let $\mathbf{u},\mathbf{x}\in\mathbb{C}^k$ with $|u_1|+\cdots+|u_k|<1$. We have then
	\begin{equation}\label{Prop-1-1}
		\Phi_1\left[\alpha,\beta;\gamma;-\langle \mathbf{u}\rangle,-\langle \mathbf{u}\circ \mathbf{x}\rangle\right]\\
		=\sum_{\mathbf{n}=\mathbf{0}}^{\infty}
		\frac{(\alpha)_{\langle \mathbf{n}\rangle}}{(\gamma)_{\langle \mathbf{n}\rangle}}
		L_{\mathbf{n}}^{(-\beta-\langle \mathbf{n}\rangle)}(\mathbf{x})
		\mathbf{u}^{\mathbf{n}}.
	\end{equation}
	In particular, when $\alpha=\gamma$, we have
	\begin{equation}\label{Prop-1-2}
		\mathrm{e}^{-\langle\mathbf{u}\circ \mathbf{x}\rangle}(1+\langle\mathbf{u}\rangle)^{-\beta}
		=\sum_{\mathbf{n}=\mathbf{0}}^{\infty}
		L_{\mathbf{n}}^{(-\beta-\langle \mathbf{n}\rangle)}(\mathbf{x})
		\mathbf{u}^{\mathbf{n}}.
	\end{equation}
\end{proposition}
\begin{proof}
	Since
	\[
	L_{\mathbf{n}}^{(-\beta-\langle\mathbf{n}\rangle)}(\mathbf{x})
	=(-1)^{\langle\mathbf{n}\rangle}\frac{(\beta)_{\langle\mathbf{n}\rangle}}{\mathbf{n}!}
	\sum_{\mathbf{j}=\mathbf{0}}^{\mathbf{n}}
	\frac{(-n_1)_{j_1}\cdots (-n_k)_{j_k}}{(1-\beta-\langle\mathbf{n}\rangle)_{\langle\mathbf{j}\rangle}}\frac{\mathbf{x}^{\mathbf{j}}}{\mathbf{j}!},
	\]
	it follows therefore that
	\begin{align}
		\sum_{\mathbf{n}=\mathbf{0}}^{\infty}
		\frac{(\alpha)_{\langle \mathbf{n}\rangle}}{(\gamma)_{\langle \mathbf{n}\rangle}}
		L_{\mathbf{n}}^{(-\beta-\langle \mathbf{n}\rangle)}(\mathbf{x})
		\mathbf{u}^{\mathbf{n}}
		&=\sum_{\mathbf{n}=\mathbf{0}}^{\infty}
		(-\mathbf{u})^{\mathbf{n}}
		\frac{(\alpha)_{\langle \mathbf{n}\rangle}}{(\gamma)_{\langle \mathbf{n}\rangle}}\frac{(\beta)_{\langle\mathbf{n}\rangle}}{\mathbf{n}!}
		\sum_{\mathbf{j}=\mathbf{0}}^{\mathbf{n}}
		\frac{(-n_1)_{j_1}\cdots (-n_k)_{j_k}}{(1-\beta-\langle\mathbf{n}\rangle)_{\langle\mathbf{j}\rangle}}\frac{\mathbf{x}^{\mathbf{j}}}{\mathbf{j}!}\notag\\
		&=\sum_{\mathbf{n},\mathbf{j}=\mathbf{0}}^{\infty}(-\mathbf{u})^{\mathbf{n}+\mathbf{j}}
		\frac{(\alpha)_{\langle \mathbf{n}\rangle+\langle \mathbf{j}\rangle}}{(\gamma)_{\langle \mathbf{n}\rangle+\langle \mathbf{j}\rangle}}\frac{(\beta)_{\langle \mathbf{n}\rangle+\langle \mathbf{j}\rangle}}{(\mathbf{n}+\mathbf{j})!}
		\frac{(-n_1-j_1)_{j_1}\cdots (-n_k-j_k)_{j_k}}{(1-\beta-\langle\mathbf{n}\rangle-\langle\mathbf{j}\rangle)_{\langle\mathbf{j}\rangle}}\frac{\mathbf{x}^{\mathbf{j}}}{\mathbf{j}!}\notag\\
		&=\sum_{\mathbf{n},\mathbf{j}=\mathbf{0}}^{\infty}\frac{(\alpha)_{\langle \mathbf{n}\rangle+\langle \mathbf{j}\rangle}(\beta)_{\langle \mathbf{n}\rangle}}{(\gamma)_{\langle \mathbf{n}\rangle+\langle \mathbf{j}\rangle}}
		\frac{(-\mathbf{u})^{\mathbf{n}}}{\mathbf{n}!}
		\frac{(-\mathbf{u}\circ\mathbf{x})^{\mathbf{j}}}{\mathbf{j}!}.
		\label{Prop-1-Proof-2}
	\end{align}
	The case $k = 1$ (i.e., the one-dimensional case) in \eqref{Prop-1-Proof-2} is trivial. However, when $k\geq 2$, the situation becomes somewhat different. 
	
	Recall that (\cite[p. 19, Eq. (2.4)]{Srivastava-1984})
	\[
	\sum_{\mathbf{n}=0}^{\infty}f(\langle\mathbf{n}\rangle)
	\frac{\mathbf{x}^{\mathbf{n}}}{\mathbf{n}!}
	=\sum_{n=0}^{\infty}f(n)
	\frac{\langle\mathbf{x}\rangle^n}{n!},
	\]
	then the multiple series in \eqref{Prop-1-Proof-2} can be further simplified.
	
	We have
	\begin{align}
		\sum_{\mathbf{n},\mathbf{j}=\mathbf{0}}^{\infty}
		\frac{(\alpha)_{\langle \mathbf{n}\rangle+\langle \mathbf{j}\rangle}(\beta)_{\langle \mathbf{n}\rangle}}{(\gamma)_{\langle \mathbf{n}\rangle+\langle \mathbf{j}\rangle}}
		\frac{(-\mathbf{u})^{\mathbf{n}}}{\mathbf{n}!}
		\frac{(-\mathbf{u}\circ\mathbf{x})^{\mathbf{j}}}{\mathbf{j}!}
		&=\sum_{\mathbf{j}=\mathbf{0}}^{\infty}
		\frac{(\alpha)_{\langle \mathbf{j}\rangle}}{(\gamma)_{\langle \mathbf{j}\rangle}}
		\frac{(-\mathbf{u}\circ\mathbf{x})^{\mathbf{j}}}{\mathbf{j}!}
		\sum_{\mathbf{n}=\mathbf{0}}^{\infty}
		\frac{(\alpha+\langle \mathbf{j}\rangle)_{\langle \mathbf{n}\rangle}(\beta)_{\langle \mathbf{n}\rangle}}{(\gamma+\langle \mathbf{j}\rangle)_{\langle \mathbf{n}\rangle}}
		\frac{(-\mathbf{u})^{\mathbf{n}}}{\mathbf{n}!}\notag\\
		&=\sum_{\mathbf{j}=\mathbf{0}}^{\infty}
		\frac{(\alpha)_{\langle \mathbf{j}\rangle}}{(\gamma)_{\langle \mathbf{j}\rangle}}
		\frac{(-\mathbf{u}\circ\mathbf{x})^{\mathbf{j}}}{\mathbf{j}!}
		\sum_{n=0}^{\infty}
		\frac{(\alpha+\langle \mathbf{j}\rangle)_{n}(\beta)_{n}}{(\gamma+\langle \mathbf{j}\rangle)_{n}}
		\frac{(-\langle\mathbf{u}\rangle)^n}{n!}\notag\\
		&=\sum_{\mathbf{j}=\mathbf{0}}^{\infty}
		\frac{(\alpha)_{\langle \mathbf{j}\rangle}}{(\gamma)_{\langle \mathbf{j}\rangle}}
		\frac{(-\mathbf{u}\circ\mathbf{x})^{\mathbf{j}}}{\mathbf{j}!}
		{}_{2}F_{1}\left[\begin{matrix}
			\alpha+\langle \mathbf{j}\rangle, \beta\\
			\gamma+\langle \mathbf{j}\rangle
		\end{matrix};-\langle\mathbf{u}\rangle\right]\notag\\
		&=\sum_{j=0}^{\infty}
		\frac{(\alpha)_{j}}{(\gamma)_{j}}
		\frac{(-\langle\mathbf{u}\circ\mathbf{x}\rangle)^{j}}{j!}
		{}_{2}F_{1}\left[\begin{matrix}
			\alpha+j, \beta\\
			\gamma+j
		\end{matrix};-\langle\mathbf{u}\rangle\right]\notag\\
		&=\Phi_1\left[\alpha,\beta;\gamma;-\langle\mathbf{u}\rangle,-\langle\mathbf{u}\circ\mathbf{x}\rangle\right].\label{Prop-1-Proof-3}
	\end{align}
	This evidently completes the proof of \eqref{Prop-1-1}. The formula \eqref{Prop-1-2} follows immediately from the reduction formula
	\[
	\Phi_1[a,b;a;x,y]=(1-x)^{-b}\mathrm{e}^{y}. 
	\]
\end{proof}

\emph{Alternative proofs of \eqref{Prop-1-1} and \eqref{Prop-1-2}.} Recall the contour integral representation (see \cite[p. 45, Eq. (12.3)]{Erdelyi-1937}; see also \cite[p. 352, Corollary 2]{Ozarslan-2014})
\begin{equation}\label{ContourIR-MLaguerre}
	L_{\mathbf{n}}^{(\alpha)}(\mathbf{x})=\frac{\Gamma(\alpha+1+\langle\mathbf{n}\rangle)}{2\pi\mathrm{i}\,  \mathbf{n}!}\int_{-\infty}^{(0+)}\mathrm{e}^s s^{-\alpha-1}\left(1-\frac{x_1}{s}\right)^{n_1}
	\cdots 
	\left(1-\frac{x_k}{s}\right)^{n_k}\mathrm{d}s,
\end{equation}
where $\alpha+1+\langle\mathbf{n}\rangle\notin\mathbb{Z}_{\leq0}$, $\int_{-\infty}^{(0+)}$ is a loop starting at $-\infty$, encircling the point $0$ in the positive sense, and returning to $-\infty$. 
Both \eqref{Prop-1-1} and \eqref{Prop-1-2} can be derived by replacing $L_{\mathbf{n}}^{(\alpha)}(\mathbf{x})$ on the right-hand side with its contour integral representation \eqref{ContourIR-MLaguerre}, and interchanging the order of summation and integration, and then evaluating the resulting integral. For example,  
\begin{align*}
	\sum_{\mathbf{n}=\mathbf{0}}^{\infty}
	L_{\mathbf{n}}^{(-\beta-\langle \mathbf{n}\rangle)}(\mathbf{x})
	\mathbf{u}^{\mathbf{n}}
	&=\frac{\Gamma(1-\beta)}{2\pi\mathrm{i}}
	\sum_{\mathbf{n}=\mathbf{0}}^{\infty}
	\frac{\mathbf{u}^{\mathbf{n}}}{\mathbf{n}!}\int_{-\infty}^{(0+)}\mathrm{e}^s s^{\beta+\langle \mathbf{n}\rangle-1}\left(1-\frac{x_1}{s}\right)^{n_1}
	\cdots 
	\left(1-\frac{x_k}{s}\right)^{n_k}\mathrm{d}s\\
	&=\frac{\Gamma(1-\beta)}{2\pi\mathrm{i}}
	\int_{-\infty}^{(0+)}\mathrm{e}^s s^{\beta-1}
	\left[\sum_{\mathbf{n}=\mathbf{0}}^{\infty}
	\frac{\mathbf{u}^{\mathbf{n}}}{\mathbf{n}!}
	(s-x_1)^{n_1}\cdots (s-x_k)^{n_k}\right]\mathrm{d}s\\
	&=\mathrm{e}^{-\langle\mathbf{u}\circ\mathbf{x}\rangle}\cdot\frac{\Gamma(1-\beta)}{2\pi\mathrm{i}}
	\int_{-\infty}^{(0+)}\mathrm{e}^{(1+\langle\mathbf{u}\rangle)s} s^{\beta-1}
	\mathrm{d}s=\mathrm{e}^{-\langle\mathbf{u}\circ \mathbf{x}\rangle}(1+\langle\mathbf{u}\rangle)^{-\beta},
\end{align*}
where we have used the following contour integral representation for the reciprocal gamma function \cite[p. 139, Eq. (5.9.2)]{NIST Handbook}:
\[
\frac{1}{\Gamma(z)}=\frac{1}{2\pi\mathrm{i}}\int_{-\infty}^{(0+)}\mathrm{e}^{s}s^{-z}\mathrm{d}s.
\]

It may be pointed out here that \eqref{Prop-1-1} can also be derived by suitably integrating \eqref{Prop-1-2}. In fact, by letting $\mathbf{u}=(u_1,\cdots,u_k)\rightarrow (tu_1,\cdots,tu_k)=t\mathbf{u}$ $(t\in(0,1))$ in \eqref{Prop-1-2}, we have then
\begin{equation}\label{Prop-1-Proof-4}
	\mathrm{e}^{-t\langle \mathbf{u}\circ \mathbf{x}\rangle}(1+t\langle \mathbf{u}\rangle)^{-\beta}
	=\sum_{\mathbf{n}=\mathbf{0}}^{\infty}
	L_{\mathbf{n}}^{(-\beta-\langle \mathbf{n}\rangle)}(\mathbf{x})
	\mathbf{u}^{\mathbf{n}}t^{\langle\mathbf{n}\rangle}.
\end{equation}
Integrating the above identity with respect to the measure
\[
\mathrm{d}\mu_{\alpha,\gamma-\alpha}(t):=\frac{\Gamma(\gamma)}{\Gamma(\alpha)\Gamma(\gamma-\alpha)}t^{\alpha-1}(1-t)^{\gamma-\alpha-1}\mathrm{d}t~~~
(\Re(\gamma)>\Re(\alpha)>0)
\]
over $(0,1)$, we obtain
\begin{equation}\label{Prop-1-Proof-5}
	\int_0^1\mathrm{e}^{-t\langle \mathbf{u}\circ \mathbf{x}\rangle}(1+t\langle \mathbf{u}\rangle)^{-\beta}
	\mathrm{d}\mu_{\alpha,\gamma-\alpha}(t)
	=\sum_{\mathbf{n}=\mathbf{0}}^{\infty}
	\frac{(\alpha)_{\langle\mathbf{n}\rangle}}{(\gamma)_{\langle\mathbf{n}\rangle}}
	L_{\mathbf{n}}^{(-\beta-\langle \mathbf{n}\rangle)}(\mathbf{x})
	\mathbf{u}^{\mathbf{n}}.
\end{equation}
The left-hand side of \eqref{Prop-1-Proof-5} can at once be interpreted as the $\Phi_1$ function in view of the integral representation \cite[p. 3, Eq. (2.3)]{Hang-Hu-Luo-2026}:
\begin{equation}\label{Prop-1-Proof-6}
	\Phi_1[\alpha,\beta,\gamma;x,y]=\int_0^1(1-xt)^{-\beta}\mathrm{e}^{yt}\mathrm{d}\mu_{\alpha,\gamma-\alpha}(t), ~~~ \Re(\gamma)>\Re(\alpha)>0.
\end{equation}

\subsection{Special cases}

An easy but elegant result can be obtained by using the decomposition technique \cite[Chapter 3]{Srivastava-Manocha-Book-1984}. 
\begin{corollary}\label{Prop-1-Cor-0}
	Let $-\beta\in\mathbb{R}\setminus\mathbb{Z}_{\leq 0}$ and $r_j>0$ $(j=1,\cdots,k)$ with $\langle\mathbf{r}\rangle<1$. Also, let $\mathbf{x}\in\mathbb{R}^k$ and $\theta_j\in(-\pi,\pi]$ $(j=1,\cdots,k)$. We have
	\begin{align}\label{Prop-1-Cor-0-1}
		&\sum_{\mathbf{n}=\mathbf{0}}^{\infty}
		L_{\mathbf{n}}^{(-\beta-\langle \mathbf{n}\rangle)}(\mathbf{x})
		\mathbf{r}^{\mathbf{n}}
		\cos\langle\mathbf{n}\circ\bm{\theta}\rangle\notag\\
		&\hspace{0.5cm}=(S(\bm{\theta}))^{-\beta}\exp\left(-\sum_{j=1}^{k}x_jr_j\cos\theta_j\right)
		\cos\left(\sum_{j=1}^{k}x_jr_j\sin\theta_j+\beta\Phi(\bm{\theta})\right)
	\end{align}
	and
	\begin{align}\label{Prop-1-Cor-0-2}
		&\sum_{\mathbf{n}=\mathbf{0}}^{\infty}
		L_{\mathbf{n}}^{(-\beta-\langle \mathbf{n}\rangle)}(\mathbf{x})
		\mathbf{r}^{\mathbf{n}}
		\sin\langle\mathbf{n}\circ\bm{\theta}\rangle\notag\\
		&\hspace{0.5cm}=-(S(\bm{\theta}))^{-\beta}\exp\left(-\sum_{j=1}^{k}x_jr_j\cos\theta_j\right)
		\sin\left(\sum_{j=1}^{k}x_j\sin\theta_j+\beta\Phi(\bm{\theta})\right),
	\end{align}
	where
	\begin{equation}\label{Prop-1-Cor-0-3}
		S(\bm{\theta}):=\left|1+\sum_{j=1}^{k}r_j\mathrm{e}^{\mathrm{i}\theta_j}\right|,\quad 
		\Phi(\bm{\theta}):=\arg\left(1+\sum_{j=1}^{k}r_j\mathrm{e}^{\mathrm{i}\theta_j}\right).
	\end{equation}
\end{corollary}
\begin{proof}
	Letting $u_j=r_j\mathrm{e}^{\mathrm{i}\theta_j}$ $(j=1,\cdots,k)$ in \eqref{Prop-1-2}, we obtain
	\begin{align}\label{Prop-1-Cor-0-Proof-1}
		\exp\left(-\sum_{j=1}^{k}x_j r_j \mathrm{e}^{\mathrm{i}\theta_j}\right)\left(1+\sum_{j=1}^{k}r_j\mathrm{e}^{\mathrm{i}\theta_j}\right)^{-\beta}
		&=\sum_{\mathbf{n}=\mathbf{0}}^{\infty}
		L_{\mathbf{n}}^{(-\beta-\langle \mathbf{n}\rangle)}(\mathbf{x})
		\mathbf{r}^{\mathbf{n}}\mathrm{e}^{\mathrm{i}\langle\mathbf{n}\circ\bm{\theta}\rangle}\notag\\ 
		&=\sum_{\mathbf{n}=\mathbf{0}}^{\infty}
		L_{\mathbf{n}}^{(-\beta-\langle \mathbf{n}\rangle)}(\mathbf{x})
		\mathbf{r}^{\mathbf{n}}(\cos\langle\mathbf{n}\circ\bm{\theta}\rangle+\mathrm{i}\sin\langle\mathbf{n}\circ\bm{\theta}\rangle).
	\end{align}
	If we write
	\begin{equation}
		\left(1+\sum_{j=1}^{k}r_j\mathrm{e}^{\mathrm{i}\theta_j}\right)^{-\beta}
		=(S(\bm{\theta}))^{-\beta}\exp\left(-\mathrm{i}\beta\Phi(\bm{\theta})\right), 
	\end{equation}	
	then the expression on the left side of \eqref{Prop-1-Cor-0-Proof-1} gives
	\begin{align}\label{Prop-1-Cor-0-Proof-2}
		&\exp\left(-\sum_{j=1}^{k}x_jr_j\mathrm{e}^{\mathrm{i}\theta_j}\right)\left(1+\sum_{j=1}^{k}r_j\mathrm{e}^{\mathrm{i}\theta_j}\right)^{-\beta}\notag\\
		&\hspace{0.5cm}=(S(\bm{\theta}))^{-\beta}\exp\left(-\sum_{j=1}^{k}x_jr_j\cos\theta_j\right)\exp\left(-\mathrm{i}\sum_{j=1}^{k}x_jr_j\sin\theta_j-\mathrm{i}\beta\Phi(\bm{\theta})\right).
	\end{align}
	Using \eqref{Prop-1-Cor-0-Proof-2} in \eqref{Prop-1-Cor-0-Proof-1} and taking the real and imaginary parts on both sides of the resulting equation, we obtain the desired results \eqref{Prop-1-Cor-0-1} and \eqref{Prop-1-Cor-0-2}, respectively.
\end{proof}

In particular, when $k=1$ and $r=1$ in \eqref{Prop-1-Cor-0-3}, we have 
\[
S(\theta)=2\cos\frac{\theta}{2},\quad \Phi(\theta)=\frac{\theta}{2}, 
\]
and consequently we obtain therefore from the results \eqref{Prop-1-Cor-0-1} and \eqref{Prop-1-Cor-0-2} of Corollary \ref{Prop-1-Cor-0}, the following known familiar results \cite[p. 209, Eqs. (11) and (12)]{Srivastava-Manocha-Book-1984}:
\begin{equation}\label{Prop-1-Cor-0-4}
	\sum_{n=0}^{\infty}
	L_{n}^{(-\beta-n)}(x)
	\cos n\theta
	=\left(2\cos\frac{\theta}{2}\right)^{-\beta}\mathrm{e}^{-x\cos\theta}\cos\left(x\sin\theta+\frac{\beta\theta}{2}\right)
\end{equation}
and
\begin{equation}\label{Prop-1-Cor-0-5}
	\sum_{n=0}^{\infty}
	L_{n}^{(-\beta-n)}(x)
	\sin n\theta
	=-\left(2\cos\frac{\theta}{2}\right)^{-\beta}\mathrm{e}^{-x\cos\theta}\sin\left(x\sin\theta+\frac{\beta\theta}{2}\right).
\end{equation}

We deem it appropriate to examine the convergence analysis for \eqref{Prop-1-Cor-0-4} and \eqref{Prop-1-Cor-0-5} in order to justify the setting of the substitution $r=1$ (see above), as this is not addressed in \cite{Srivastava-Manocha-Book-1984}. Since
\[
L_n^{(-\beta-n)}(z)=(-1)^n\frac{(\beta)_n}{n!} {}_{1}F_{1}\left[\begin{matrix}
	-n\\
	1-\beta-n
\end{matrix};x\right],
\]
Tannery's theorem \cite[p. 243]{Wang-2010} gives immediately 
\[
L_n^{(-\beta-n)}(x)\sim (-1)^{n}\frac{(\beta)_n}{n!}\mathrm{e}^{x}, \quad n\rightarrow+\infty,
\]
and therefore both series \eqref{Prop-1-Cor-0-4} and \eqref{Prop-1-Cor-0-5} converge for all $\theta\in(-\pi,\pi]$ under the condition that $\beta<0$. A natural question is whether one can set $\mathbf{r}=\mathbf{1}$ in the Corollary \ref{Prop-1-Cor-0} as was done to obtain the results \eqref{Prop-1-Cor-0-3} and \eqref{Prop-1-Cor-0-4}. This may require an accurate estimate for $L_{\mathbf{n}}^{(-\beta-\langle\mathbf{n}\rangle)}(\mathbf{x})$. A more particular and specific problem is given below in Section \ref{Sec-ConcludingRemarks}.

\subsection{Some applications}

A simple application of \eqref{Prop-1-2} yields the multiplication formula, which generalizes a formula of Carlitz (see \cite[p. 223]{Carlitz-1960}). 
\begin{corollary}\label{Prop-1-Cor-1}
	Let $\lambda\in\mathbb{C}$, $-\beta\in\mathbb{C}\setminus\mathbb{Z}_{\leq0}$ and $\mathbf{x}\in\mathbb{C}^k$. Then we have
	\begin{equation}\label{Prop-1-Cor-1-1}
		L_{\mathbf{m}}^{(-\beta-\langle \mathbf{m}\rangle)}(\lambda\mathbf{x})
		=\sum_{\mathbf{n}=\mathbf{0}}^{\mathbf{m}}\frac{(1-\lambda)^{\langle \mathbf{m}\rangle-\langle\mathbf{n}\rangle}\mathbf{x}^{\mathbf{m}-\mathbf{n}}}{(\mathbf{m}-\mathbf{n})!}L_{\mathbf{n}}^{(-\beta-\langle \mathbf{n}\rangle)}(\mathbf{x}).
	\end{equation}
\end{corollary}
\begin{proof}
	Note that
	\begin{align*}
		\sum_{\mathbf{m}=\mathbf{0}}^{\infty}
		L_{\mathbf{m}}^{(-\beta-\langle \mathbf{m}\rangle)}(\lambda\mathbf{x})
		\mathbf{u}^{\mathbf{m}}
		&=\mathrm{e}^{-\lambda\langle\mathbf{u}\circ \mathbf{x}\rangle}(1+\langle\mathbf{u}\rangle)^{-\beta}\\
		&=\mathrm{e}^{(1-\lambda)\langle\mathbf{u}\circ \mathbf{x}\rangle}\cdot \mathrm{e}^{-\langle\mathbf{u}\circ \mathbf{x}\rangle}(1+\langle\mathbf{u}\rangle)^{-\beta}\\
		&=\sum_{\mathbf{m}=\mathbf{0}}^{\infty}\sum_{\mathbf{n}=\mathbf{0}}^{\infty}\frac{(1-\lambda)^{\langle \mathbf{m}\rangle}\mathbf{x}^{\mathbf{m}}}{\mathbf{m}!}L_{\mathbf{n}}^{(-\beta-\langle \mathbf{n}\rangle)}(\mathbf{x})\mathbf{u}^{\mathbf{m}+\mathbf{n}}\\
		&=\sum_{\mathbf{m}=\mathbf{0}}^{\infty}\left(\sum_{\mathbf{n}=\mathbf{0}}^{\mathbf{m}}\frac{(1-\lambda)^{\langle \mathbf{m}\rangle-\langle\mathbf{n}\rangle}\mathbf{x}^{\mathbf{m}-\mathbf{n}}}{(\mathbf{m}-\mathbf{n})!}L_{\mathbf{n}}^{(-\beta-\langle \mathbf{n}\rangle)}(\mathbf{x})\right)\mathbf{u}^{\mathbf{m}}.
	\end{align*}
	Then the result follows at once by comparing the coefficients.
\end{proof}

The next corollary is not only more appealing in its form but also possesses significant applications.
\begin{corollary}\label{Prop-1-Cor-2}
	Let $-\beta\in\mathbb{C}\setminus\mathbb{Z}_{\leq0}$, and let $\mathbf{u},\mathbf{x}\in\mathbb{C}^k$ with $|u_1|+\cdots+|u_k|<1$. We have
	\begin{equation}\label{Prop-1-Cor-2-1}
		\sum_{\mathbf{n}=\mathbf{0}}^{\infty}
		L_{\mathbf{n}+\mathbf{m}}^{(-\beta-\langle \mathbf{n}\rangle-\langle\mathbf{m}\rangle)}(\mathbf{x})
		\frac{(\mathbf{n}+\mathbf{m})!}{\mathbf{m}!\mathbf{n}!}\mathbf{u}^{\mathbf{n}}
		=
		\mathrm{e}^{-\langle\mathbf{u}\circ \mathbf{x}\rangle}(1+\langle\mathbf{u}\rangle)^{-\beta-\langle\mathbf{m}\rangle}L_{\mathbf{m}}^{(-\beta-\langle \mathbf{m}\rangle)}\left((1+\langle\mathbf{u}\rangle)\mathbf{x}\right).
	\end{equation}
\end{corollary}
\begin{proof}
	The proof is based on Rainville's method \cite[pp. 13--15]{MCBride-Book-1971}. First, we replace $\mathbf{u}$ with $\mathbf{u}+\mathbf{v}$ in \eqref{Prop-1-2} to obtain
	\begin{equation}\label{Prop-1-Cor-2-Proof-1}
		\mathrm{e}^{-\langle\mathbf{u}\circ \mathbf{x}\rangle-\langle\mathbf{v}\circ \mathbf{x}\rangle}(1+\langle\mathbf{u}\rangle+\langle\mathbf{v}\rangle)^{-\beta}
		=\sum_{\mathbf{n}=\mathbf{0}}^{\infty}
		L_{\mathbf{n}}^{(-\beta-\langle \mathbf{n}\rangle)}(\mathbf{x})
		(\mathbf{u}+\mathbf{v})^{\mathbf{n}}.
	\end{equation}
	Here, we assume $|u_1|+|v_1|+\cdots+|u_k|+|v_k|<1$ to guarantee convergence. 
	The right hand side of \eqref{Prop-1-Cor-2-Proof-1} can be written as
	\begin{align}\label{Prop-1-Cor-2-Proof-2}
		\sum_{\mathbf{n}=\mathbf{0}}^{\infty}
		L_{\mathbf{n}}^{(-\beta-\langle \mathbf{n}\rangle)}(\mathbf{x})
		(\mathbf{u}+\mathbf{v})^{\mathbf{n}}
		&=\sum_{\mathbf{n}=\mathbf{0}}^{\infty}
		L_{\mathbf{n}}^{(-\beta-\langle \mathbf{n}\rangle)}(\mathbf{x})
		\sum_{\mathbf{m}=\mathbf{0}}^{\mathbf{n}}\frac{\mathbf{n}!}{\mathbf{m}!(\mathbf{n}-\mathbf{m})!}\mathbf{u}^{\mathbf{n}-\mathbf{m}}\mathbf{v}^{\mathbf{m}}\notag\\
		&=\sum_{\mathbf{m}=\mathbf{0}}^{\infty}\left(\sum_{\mathbf{n}=\mathbf{0}}^{\infty}
		L_{\mathbf{n}+\mathbf{m}}^{(-\beta-\langle \mathbf{n}\rangle-\langle\mathbf{m}\rangle)}(\mathbf{x})
		\frac{(\mathbf{n}+\mathbf{m})!}{\mathbf{m}!\mathbf{n}!}\mathbf{u}^{\mathbf{n}}\right)\mathbf{v}^{\mathbf{m}}.
	\end{align}
	The left hand side of \eqref{Prop-1-Cor-2-Proof-1} can be regrouped as
	\begin{align*}
		\mathrm{e}^{-\langle\mathbf{u}\circ \mathbf{x}\rangle-\langle\mathbf{v}\circ \mathbf{x}\rangle}(1+\langle\mathbf{u}\rangle+\langle\mathbf{v}\rangle)^{-\beta}
		&=\mathrm{e}^{-\langle\mathbf{u}\circ \mathbf{x}\rangle}
		(1+\langle\mathbf{u}\rangle)^{-\beta}
		\cdot 
		\mathrm{e}^{-\langle\mathbf{v}\circ \mathbf{x}\rangle}
		\left(1+\frac{\langle\mathbf{v}\rangle}{1+\langle\mathbf{u}\rangle}\right)^{-\beta}\\
		&=\mathrm{e}^{-\langle\mathbf{u}\circ \mathbf{x}\rangle}
		(1+\langle\mathbf{u}\rangle)^{-\beta}
		\cdot 
		\mathrm{e}^{-\big\langle\frac{\mathbf{v}}{1+\langle\mathbf{u}\rangle}\circ((1+\langle\mathbf{u}\rangle)\mathbf{x})\big\rangle}
		\left(1+\left\langle\frac{\mathbf{v}}{1+\langle\mathbf{u}\rangle}\right\rangle\right)^{-\beta},
	\end{align*}
	and thus by using the generating function \eqref{Prop-1-2} again, we have
	\begin{align}\label{Prop-1-Cor-2-Proof-3}
		\mathrm{e}^{-\langle\mathbf{u}\circ \mathbf{x}\rangle-\langle\mathbf{v}\circ \mathbf{x}\rangle}(1+\langle\mathbf{u}\rangle+\langle\mathbf{v}\rangle)^{-\beta}
		&=\mathrm{e}^{-\langle\mathbf{u}\circ \mathbf{x}\rangle}
		(1+\langle\mathbf{u}\rangle)^{-\beta}
		\sum_{\mathbf{m}=\mathbf{0}}^{\infty}
		L_{\mathbf{m}}^{(-\beta-\langle \mathbf{m}\rangle)}((1+\langle\mathbf{u}\rangle)\mathbf{x})
		\mathbf{v}^{\mathbf{m}}
		(1+\langle\mathbf{u}\rangle)^{-\langle\mathbf{m}\rangle}\notag\\
		&=\sum_{\mathbf{m}=\mathbf{0}}^{\infty}
		\left(L_{\mathbf{m}}^{(-\beta-\langle \mathbf{m}\rangle)}((1+\langle\mathbf{u}\rangle)\mathbf{x})
		\mathrm{e}^{-\langle\mathbf{u}\circ \mathbf{x}\rangle}(1+\langle\mathbf{u}\rangle)^{-\beta-\langle\mathbf{m}\rangle}\right) 
		\mathbf{v}^{\mathbf{m}}.
	\end{align}
	The result \eqref{Prop-1-Cor-2-1} then follows by equating the coefficients of \eqref{Prop-1-Cor-2-Proof-2} and \eqref{Prop-1-Cor-2-Proof-3} in view of \eqref{Prop-1-Cor-2-Proof-1}.
\end{proof}	
\begin{remark}
	When $k=1$, \eqref{Prop-1-Cor-2-1} reduces to
	\begin{equation}\label{Prop-1-Cor-2-Remark-1}
		\sum_{n=0}^{\infty}
		L_{n+m}^{(-\beta-n-m)}(x)
		\frac{(n+m)!}{m!n!}u^{n}
		=
		\mathrm{e}^{-ux}(1+u)^{-\beta-m}L_{m}^{(-\beta-m)}\left((1+u)x\right),
	\end{equation}
	or equivalently, 
	\begin{equation}\label{Prop-1-Cor-2-Remark-2}
		\sum_{n=0}^{\infty}
		f_{n+m}^{\beta}(x)
		\frac{(n+m)!}{m!n!}u^{n}
		=
		\mathrm{e}^{ux}(1-u)^{-\beta-m}
		f_{m}^{\beta}\left((1-u)x\right),
	\end{equation}
	where $f_n^{\beta}(z):=(-1)^n L_n^{(-\beta-n)}(z)$ (see \cite[p. 15, Eq. (4)]{MCBride-Book-1971}). If we let $\beta\rightarrow-\beta-m$, then \eqref{Prop-1-Cor-2-Remark-1} becomes
	\begin{equation}\label{Prop-1-Cor-2-Remark-3}
		\sum_{n=0}^{\infty}
		L_{n+m}^{(\beta-n)}(x)
		\frac{(n+m)!}{m!n!}u^{n}
		=
		\mathrm{e}^{-ux}(1+u)^{\beta}L_{m}^{(\beta)}\left((1+u)x\right),
	\end{equation}
	where $\beta\in\mathbb{C}\setminus\mathbb{Z}_{\leq0}$. 
	There are some different proofs of \eqref{Prop-1-Cor-2-Remark-3} (see, for example, \cite{Carlitz-1960} and \cite{Mittal-1977}).
\end{remark}	

Finally, let us consider the series
\begin{align*}
	S:=\sum_{\mathbf{n}=\mathbf{0}}^{\infty}
	\mathbf{n}! 
	L_{\mathbf{n}}^{(-\alpha-\langle \mathbf{n}\rangle)}(\mathbf{x})
	L_{\mathbf{n}}^{(-\beta-\langle \mathbf{n}\rangle)}(\mathbf{y})
	\mathbf{u}^{\mathbf{n}}.
\end{align*}
Since
\begin{align*}
	L_{\mathbf{n}}^{(-\beta-\langle\mathbf{n}\rangle)}(\mathbf{y})
	&=(-1)^{\langle\mathbf{n}\rangle}\sum_{\mathbf{m}=\mathbf{0}}^{\mathbf{n}}
	\frac{(\beta)_{\langle\mathbf{n}\rangle-\langle\mathbf{m}\rangle}}{(\mathbf{n}-\mathbf{m})!\mathbf{m}!}
	\mathbf{y}^{\mathbf{m}},
\end{align*}
we have
\begin{align}
	S&=\sum_{\mathbf{n}=\mathbf{0}}^{\infty}
	\mathbf{n}! 
	L_{\mathbf{n}}^{(-\alpha-\langle \mathbf{n}\rangle)}(\mathbf{x})
	(-1)^{\langle\mathbf{n}\rangle}\sum_{\mathbf{m}=\mathbf{0}}^{\mathbf{n}}
	\frac{(\beta)_{\langle\mathbf{n}\rangle-\langle\mathbf{m}\rangle}}{(\mathbf{n}-\mathbf{m})!\mathbf{m}!}
	\mathbf{y}^{\mathbf{m}}
	\mathbf{u}^{\mathbf{n}}\notag\\
	&=\sum_{\mathbf{n}=\mathbf{0}}^{\infty}
	(-1)^{\langle\mathbf{n}\rangle}
	(\beta)_{\langle\mathbf{n}\rangle}
	\mathbf{u}^{\mathbf{n}}
	\sum_{\mathbf{m}=\mathbf{0}}^{\infty}
	L_{\mathbf{n}+\mathbf{m}}^{(-\alpha-\langle \mathbf{n}\rangle-\langle\mathbf{m}\rangle)}(\mathbf{x})
	\frac{(\mathbf{n}+\mathbf{m})! }{\mathbf{n}!\mathbf{m}!}
	(-\mathbf{u}\circ\mathbf{y})^{\mathbf{m}}.
	\label{Th-1-Proof-1}
\end{align}
Applying Corollary \ref{Prop-1-Cor-2} to the inner series of \eqref{Th-1-Proof-1} gives immediately
\begin{align}\label{Th-1-Proof-2}
	S&=\sum_{\mathbf{n}=\mathbf{0}}^{\infty}
	(-1)^{\langle\mathbf{n}\rangle}
	(\beta)_{\langle\mathbf{n}\rangle}
	\mathbf{u}^{\mathbf{n}}
	\cdot \mathrm{e}^{\langle\mathbf{u}\circ\mathbf{y}\circ \mathbf{x}\rangle}(1-\langle\mathbf{u}\circ\mathbf{y}\rangle)^{-\alpha-\langle\mathbf{n}\rangle}L_{\mathbf{n}}^{(-\alpha-\langle \mathbf{n}\rangle)}\left((1-\langle\mathbf{u}\circ\mathbf{y}\rangle)\mathbf{x}\right)\notag\\
	&=\mathrm{e}^{\langle\mathbf{u}\circ\mathbf{y}\circ \mathbf{x}\rangle}
	(1-\langle\mathbf{u}\circ\mathbf{y}\rangle)^{-\alpha}
	\sum_{\mathbf{n}=\mathbf{0}}^{\infty}
	(-1)^{\langle\mathbf{n}\rangle}
	(\beta)_{\langle\mathbf{n}\rangle}
	\mathbf{u}^{\mathbf{n}}
	\cdot (1-\langle\mathbf{u}\circ\mathbf{y}\rangle)^{-\langle\mathbf{n}\rangle}L_{\mathbf{n}}^{(-\alpha-\langle \mathbf{n}\rangle)}\left((1-\langle\mathbf{u}\circ\mathbf{y}\rangle)\mathbf{x}\right).
\end{align}
If $\beta=-N\in\mathbb{Z}_{\leq 0}$ in \eqref{Th-1-Proof-2}, the series terminates and we obtain
\begin{align}
	S&=\mathrm{e}^{\langle\mathbf{u}\circ\mathbf{y}\circ \mathbf{x}\rangle}
	(1-\langle\mathbf{u}\circ\mathbf{y}\rangle)^{-\alpha}\notag\\
	&\hspace{0.5cm}\cdot\sum_{\mathbf{n}=\mathbf{0}}^{\langle\mathbf{n}\rangle\leq N}
	(-1)^{\langle\mathbf{n}\rangle}
	(-N)_{\langle\mathbf{n}\rangle}
	L_{\mathbf{n}}^{(-\alpha-\langle \mathbf{n}\rangle)}\left((1-\langle\mathbf{u}\circ\mathbf{y}\rangle)\mathbf{x}\right)
	\cdot\mathbf{u}^{\mathbf{n}}
	(1-\langle\mathbf{u}\circ\mathbf{y}\rangle)^{-\langle\mathbf{n}\rangle}.
	\label{Th-1-Proof-3}
\end{align}
However, the finite sum in \eqref{Th-1-Proof-3} cannot be evaluated using Corollary \ref{Prop-1-Cor-1} unless $k=1$. The case $k=1$ has already been treated by Carlitz \cite{Carlitz-1960}. 

If $\beta\in\mathbb{C}\setminus\mathbb{Z}_{\leq 0}$ in \eqref{Th-1-Proof-2}, the series does not converge in general. However, it is not difficult to show (by adopting a formal procedure of computation) that 
\begin{align*}
	\sum_{\mathbf{n}=\mathbf{0}}^{\infty}
	(-1)^{\langle\mathbf{n}\rangle}
	(\beta)_{\langle\mathbf{n}\rangle}
	L_{\mathbf{n}}^{(-\alpha-\langle \mathbf{n}\rangle)}(\mathbf{y})\mathbf{u}^{\mathbf{n}}
	&\cong\sum_{\mathbf{n}=\mathbf{0}}^{\infty}
	\sum_{\mathbf{m}=\mathbf{0}}^{\infty}
	(\beta)_{\langle\mathbf{n}\rangle+\langle\mathbf{m}\rangle}
	\frac{(\alpha)_{\langle\mathbf{n}\rangle}}{\mathbf{n}!\mathbf{m}!}
	\mathbf{y}^{\mathbf{m}}\mathbf{u}^{\mathbf{n}+\mathbf{m}}\\
	&\cong\sum_{\mathbf{n}=\mathbf{0}}^{\infty}
	\frac{(\alpha)_{\langle\mathbf{n}\rangle}(\beta)_{\langle\mathbf{n}\rangle}}{\mathbf{n}!}
	\mathbf{u}^{\mathbf{n}}
	\sum_{\mathbf{m}=\mathbf{0}}^{\infty}
	\frac{(\beta+\langle\mathbf{n}\rangle)_{\langle\mathbf{m}\rangle}}{\mathbf{m}!}
	(\mathbf{y}\circ\mathbf{u})^{\mathbf{m}}\\
	&\cong\sum_{\mathbf{n}=\mathbf{0}}^{\infty}
	\frac{(\alpha)_{\langle\mathbf{n}\rangle}(\beta)_{\langle\mathbf{n}\rangle}}{\mathbf{n}!}
	\mathbf{u}^{\mathbf{n}}
	(1-\langle\mathbf{y}\circ\mathbf{u}\rangle)^{-\beta-\mathbf{n}}\\
	&\cong (1-\langle\mathbf{y}\circ\mathbf{u}\rangle)^{-\beta}\sum_{\mathbf{n}=\mathbf{0}}^{\infty}
	\frac{(\alpha)_{\langle\mathbf{n}\rangle}(\beta)_{\langle\mathbf{n}\rangle}}{\mathbf{n}!}
	\mathbf{u}^{\mathbf{n}}
	(1-\langle\mathbf{y}\circ\mathbf{u}\rangle)^{-\mathbf{n}}\\
	&\cong (1-\langle\mathbf{y}\circ\mathbf{u}\rangle)^{-\beta}
	\sum_{n=0}^{\infty}\frac{(\alpha)_n(\beta)_n}{n!}\left(\frac{\langle\mathbf{u}\rangle}{1-\langle\mathbf{y}\circ\mathbf{u}\rangle}\right)^n\\
	&\cong (1-\langle\mathbf{y}\circ\mathbf{u}\rangle)^{-\beta}
	{}_{2}F_{0}\left[\begin{matrix}
		\alpha,\beta\\
		-
	\end{matrix}\frac{\langle\mathbf{u}\rangle}{1-\langle\mathbf{y}\circ\mathbf{u}\rangle}\right].
\end{align*}
where the symbol $\cong$ means ``formally equal''. We thus obtain from \eqref{Th-1-Proof-2}: 
\begin{equation}\label{Th-1-Proof-4} 
	S\cong\mathrm{e}^{\langle\mathbf{u}\circ\mathbf{y}\circ \mathbf{x}\rangle}
	(1-\langle\mathbf{u}\circ\mathbf{y}\rangle)^{-\alpha}
	(1-\langle\mathbf{u}\circ\mathbf{x}\rangle)^{-\beta}
	{}_{2}F_{0}\left[\begin{matrix}
		\alpha,\beta\\
		-
	\end{matrix}\frac{\langle\mathbf{u}\rangle}{(1-\langle\mathbf{x}\circ\mathbf{u}\rangle)(1-\langle\mathbf{y}\circ\mathbf{u}\rangle)}\right].
\end{equation}
The relation \eqref{Th-1-Proof-4} is a multivariate generalization of the following bilinear generating relation (\cite[p. 23, Eq. (5a)]{MCBride-Book-1971}; see also \cite[p. 52, Eq. (4.21)]{Mittal-1977}):
\[ 
\sum_{n=0}^{\infty}n!f_n^{\alpha}(x) f_n^{\beta}(y)t^n
\cong \mathrm{e}^{xyt}(1-yt)^{-\alpha}(1-xt)^{-\beta}{}_{2}F_{0}\left[\begin{matrix}
	\alpha,\beta\\
	-
\end{matrix};\frac{t}{(1-xt)(1-yt)}\right],
\] 
where $f_n^\alpha(x)$ is the modified Laguerre polynomials.

\section{Generating functions of type II}\label{Sec-GF-Type-II}

\subsection{Main results}

Let $\beta_r-\beta_{r-1}\in\mathbb{C}\setminus\mathbb{Z}_{\leq0}$ for $r=1,\cdots,k$; where $\beta_0\equiv\beta$. Also, let $\alpha,\sigma_1,\cdots,\sigma_k\in\mathbb{C}$, $\gamma\in\mathbb{C}\setminus\mathbb{Z}_{\leq 0}$, $x\in\mathbb{C}$ $(|x|<1)$ and $y\in\mathbb{C}$. Tremblay and Lavertu \cite[p. 13, Eq. (2.4)]{Tremblay-Lavertu-1972} proved that
\begin{align}\label{GF-TremblayLavertu-0}
	\Phi_1\left[\alpha,\beta;\gamma;x,y\right]
	&=\sum_{r_1,\cdots,r_k=0}^{\infty}\frac{(\alpha)_{r_1+\cdots+r_k}}{(\gamma)_{r_1+\cdots+r_k}} (-x)^{r_1+\cdots+r_k}\notag\\
	&\hspace{1cm}\cdot L_{r_1}^{(\beta_1-\beta-r_1)}\left((1-\sigma_1)y/x\right) \cdots 
	L_{r_k}^{(\beta_k-\beta_{k-1}-r_k)}\left((1-\sigma_k)\sigma_1\cdots\sigma_{k-1}y/x\right)\notag\\
	&\hspace{1cm}\cdot\Phi_1\left[\alpha+r_1+\cdots+r_k,\beta_k;\gamma+r_1+\cdots+r_k;x,\sigma_1\cdots\sigma_k y\right], 
\end{align}

Special cases of \eqref{GF-TremblayLavertu-0} yield some interesting formulas, and one such special case was stated by Tremblay and Lavertu as the main result \cite[p. 12, Eq. (1.4)]{Tremblay-Lavertu-1972} given by
\begin{align}\label{GF-TremblayLavertu}
	\Phi_1\left[\alpha,\beta;\gamma;-u,-u(w_1+\cdots+w_k)\right]
	&=\sum_{r_1,\cdots,r_k=0}^{\infty}\frac{(\alpha)_{r_1+\cdots+r_k}}{(\gamma)_{r_1+\cdots+r_k}} u^{r_1+\cdots+r_k}\notag\\
	&\hspace{1cm}\cdot L_{r_k}^{(-\beta_{k-1}-r_k)}(w_k)
	\prod_{j=1}^{k-1}L_{r_j}^{(\beta_j-\beta_{j-1}-r_j)}(w_j). 
\end{align}
Note that when $k=1$, \eqref{GF-TremblayLavertu} reduces to \eqref{GF-AbdulHalim-AlSalam}.

The formula \eqref{GF-TremblayLavertu-0} is of interest to us because it is quite different from the existing multiple generating functions for Laguerre polynomials, see for instance, the well-known generating function of Srivastava and Singhal \cite[p. 1239, Eq. (5)]{Srivastava-Singhal-1972}; see also \cite[p. 9, Eq. (1.13)]{Srivastava-1984}. However, the formula \eqref{GF-TremblayLavertu-0} seems to have been overlooked in subsequent researches. In this subsection, we explore the multivariate generalization \eqref{GF-TremblayLavertu-0}.

First, we prove the following lemma.

\begin{lemma}\label{Lemma}
	Let $\sigma, \alpha\in\mathbb{C}$ and $\beta_1-\beta, \gamma\in\mathbb{C}\setminus\mathbb{Z}_{\leq 0}$. 
	Also, let $\mathbf{x},\mathbf{y}\in\mathbb{C}^k$ with $|x_1|+\cdots+|x_k|<1$. We have
	\begin{align}\label{Lemma-1}
		\Phi_1\left[\alpha,\beta;\gamma;\langle\mathbf{x}\rangle,\langle\mathbf{y}\rangle\right]
		&=\sum_{\mathbf{n}=\mathbf{0}}^{\infty}\frac{(\alpha)_{\langle\mathbf{n}\rangle}}{(\gamma)_{\langle\mathbf{n}\rangle}}
		(-\mathbf{x})^{\mathbf{n}}
		L_{\mathbf{n}}^{(\beta_1-\beta-\langle\mathbf{n}\rangle)}\left((1-\bm{\sigma})\circ\mathbf{y}/\mathbf{x}\right)\notag\\
		&\hspace{1.5cm}\cdot \Phi_1\left[\alpha+\langle\mathbf{n}\rangle,\beta_1;\gamma+\langle\mathbf{n}\rangle;\langle\mathbf{x}\rangle,\langle\bm{\sigma}\circ\mathbf{y}\rangle\right].
	\end{align} 	
\end{lemma}
\begin{proof}
	We observe that
	\begin{align*}
		L_{\mathbf{n}}^{(\beta_1-\beta-\langle\mathbf{n}\rangle)}((1-\bm{\sigma})\circ\mathbf{y}/\mathbf{x})
		&=(-1)^{\langle\mathbf{n}\rangle}\frac{(\beta-\beta_1)_{\langle\mathbf{n}\rangle}}{\mathbf{n}!}
		\sum_{\mathbf{j}=\mathbf{0}}^{\mathbf{n}}
		\frac{(-n_1)_{j_1}\cdots (-n_k)_{j_k}}{(1-\beta+\beta_1-\langle\mathbf{n}\rangle)_{\langle\mathbf{j}\rangle}}\frac{(\mathbf{1}-\bm{\sigma})^{\mathbf{j}}}{\mathbf{j}!}
		\frac{\mathbf{y}^{\mathbf{j}}}{\mathbf{x}^{\mathbf{j}}}\\
		&=(-1)^{\langle\mathbf{n}\rangle}\sum_{\mathbf{j}=\mathbf{0}}^{\mathbf{n}}
		\frac{(\beta-\beta_1)_{\langle\mathbf{n}\rangle-\langle\mathbf{j}\rangle}}{(\mathbf{n}-\mathbf{j})!\mathbf{j}!}(\mathbf{1}-\bm{\sigma})^{\mathbf{j}}
		\frac{\mathbf{y}^{\mathbf{j}}}{\mathbf{x}^{\mathbf{j}}}.
	\end{align*}
	
	Following similar calculations as in \eqref{Prop-1-Proof-3}, we have then
	\begin{align}
		&\sum_{\mathbf{n}=\mathbf{0}}^{\infty}\frac{(\alpha)_{\langle\mathbf{n}\rangle}}{(\gamma)_{\langle\mathbf{n}\rangle}}
		(-\mathbf{x})^{\mathbf{n}}
		\Phi_1\left[\alpha+\langle\mathbf{n}\rangle,\beta_1;\gamma+\langle\mathbf{n}\rangle;\langle\mathbf{x}\rangle,\langle\bm{\sigma}\circ\mathbf{y}\rangle\right]
		L_{\mathbf{n}}^{(\beta_1-\beta-\langle\mathbf{n}\rangle)}\left((1-\bm{\sigma})\circ\mathbf{y}/\mathbf{x}\right)
		\notag\\
		&\hspace{1cm}=\sum_{\mathbf{n}=\mathbf{0}}^{\infty}\frac{(\alpha)_{\langle\mathbf{n}\rangle}}{(\gamma)_{\langle\mathbf{n}\rangle}}
		\mathbf{x}^{\mathbf{n}}
		\sum_{\mathbf{p},\mathbf{q}=\mathbf{0}}^{\infty}
		\frac{(\alpha+\langle\mathbf{n}\rangle)_{\langle\mathbf{p}\rangle+\langle\mathbf{q}\rangle}(\beta_1)_{\langle\mathbf{p}\rangle}}{(\gamma+\langle\mathbf{n}\rangle)_{\langle\mathbf{p}\rangle+\langle\mathbf{q}\rangle}}
		\frac{\mathbf{x}^{\mathbf{p}}}{\mathbf{p}!}
		\frac{(\bm{\sigma}\circ\mathbf{y})^{\mathbf{q}}}{\mathbf{q}!}\sum_{\mathbf{j}=\mathbf{0}}^{\mathbf{n}}
		\frac{(\beta-\beta_1)_{\langle\mathbf{n}\rangle-\langle\mathbf{j}\rangle}}{(\mathbf{n}-\mathbf{j})!\mathbf{j}!}
		(\mathbf{1}-\bm{\sigma})^{\mathbf{j}}
		\frac{\mathbf{y}^{\mathbf{j}}}{\mathbf{x}^{\mathbf{j}}}
		\notag\\
		&\hspace{1cm}=\sum_{\mathbf{n},\mathbf{j}=\mathbf{0}}^{\infty}\frac{(\alpha)_{\langle\mathbf{n}\rangle+\langle\mathbf{j}\rangle}}{(\gamma)_{\langle\mathbf{n}\rangle+\langle\mathbf{j}\rangle}}
		\mathbf{x}^{\mathbf{n}+\mathbf{j}}
		\sum_{\mathbf{p},\mathbf{q}=\mathbf{0}}^{\infty}
		\frac{(\alpha+\langle\mathbf{n}\rangle+\langle\mathbf{j}\rangle)_{\langle\mathbf{p}\rangle+\langle\mathbf{q}\rangle}(\beta_1)_{\langle\mathbf{p}\rangle}}{(\gamma+\langle\mathbf{n}\rangle+\langle\mathbf{j}\rangle)_{\langle\mathbf{p}\rangle+\langle\mathbf{q}\rangle}}
		\frac{\mathbf{x}^{\mathbf{p}}}{\mathbf{p}!}
		\frac{(\bm{\sigma}\circ \mathbf{y})^{\mathbf{q}}}{\mathbf{q}!}
		\frac{(\beta-\beta_1)_{\langle\mathbf{n}\rangle}}{\mathbf{n}!\mathbf{j}!}(\mathbf{1}-\bm{\sigma})^{\mathbf{j}}
		\frac{\mathbf{y}^{\mathbf{j}}}{\mathbf{x}^{\mathbf{j}}}\notag\\
		&\hspace{1cm}=\sum_{\mathbf{n},\mathbf{j},\mathbf{p},\mathbf{q}=\mathbf{0}}^{\infty}
		\frac{(\alpha)_{\langle\mathbf{n}\rangle+\langle\mathbf{j}\rangle+\langle\mathbf{p}\rangle+\langle\mathbf{q}\rangle}(\beta_1)_{\langle\mathbf{p}\rangle}}{(\gamma)_{\langle\mathbf{n}\rangle+\langle\mathbf{j}\rangle+\langle\mathbf{p}\rangle+\langle\mathbf{q}\rangle}}
		\frac{\mathbf{x}^{\mathbf{n}+\mathbf{p}}}{\mathbf{p}!\mathbf{q}!}
		\frac{(\beta-\beta_1)_{\langle\mathbf{n}\rangle}}{\mathbf{n}!\mathbf{j}!}
		\bm{\sigma}^{\mathbf{q}}(\mathbf{1}-\bm{\sigma})^{\mathbf{j}}\mathbf{y}^{\mathbf{q}+\mathbf{j}}\notag\\
		&\hspace{1cm}=\sum_{\mathbf{n},\mathbf{p},\mathbf{j}=\mathbf{0}}^{\infty}
		\frac{(\alpha)_{\langle\mathbf{n}\rangle+\langle\mathbf{j}\rangle+\langle\mathbf{p}\rangle}(\beta_1)_{\langle\mathbf{p}\rangle}}{(\gamma)_{\langle\mathbf{n}\rangle+\langle\mathbf{j}\rangle+\langle\mathbf{p}\rangle}}
		\frac{\mathbf{x}^{\mathbf{p}+\mathbf{n}}}{\mathbf{p}!}
		\frac{(\beta-\beta_1)_{\langle\mathbf{n}\rangle}}{\mathbf{n}!}
		(\mathbf{1}-\bm{\sigma})^{\mathbf{j}}
		\mathbf{y}^{\mathbf{j}}\sum_{\mathbf{q}=\mathbf{0}}^{\mathbf{j}}\frac{(\bm{\sigma}/(\mathbf{1}-\bm{\sigma}))^{\mathbf{q}}}{\mathbf{q}!(\mathbf{j}-\mathbf{q})!}
		~~(\mathbf{j}\rightarrow \mathbf{j}-\mathbf{q})\notag\\
		&\hspace{1cm}=\sum_{\mathbf{n},\mathbf{p},\mathbf{j}=\mathbf{0}}^{\infty}
		\frac{(\alpha)_{\langle\mathbf{n}\rangle+\langle\mathbf{j}\rangle+\langle\mathbf{p}\rangle}(\beta_1)_{\langle\mathbf{p}\rangle}}{(\gamma)_{\langle\mathbf{n}\rangle+\langle\mathbf{j}\rangle+\langle\mathbf{p}\rangle}}
		\frac{\mathbf{x}^{\mathbf{p}+\mathbf{n}}}{\mathbf{p}!}
		\frac{(\beta-\beta_1)_{\langle\mathbf{n}\rangle}}{\mathbf{n}!}
		\frac{(\mathbf{1}-\bm{\sigma})^{\mathbf{j}}\mathbf{y}^{\mathbf{j}}}{\mathbf{j}!}\frac{1}{(\mathbf{1}-\bm{\sigma})^{\mathbf{j}}}\notag\\
		&\hspace{1cm}=\sum_{\mathbf{n},\mathbf{p},\mathbf{j}=\mathbf{0}}^{\infty}
		\frac{(\alpha)_{\langle\mathbf{n}\rangle+\langle\mathbf{j}\rangle+\langle\mathbf{p}\rangle}(\beta_1)_{\langle\mathbf{p}\rangle}}{(\gamma)_{\langle\mathbf{n}\rangle+\langle\mathbf{j}\rangle+\langle\mathbf{p}\rangle}}
		\frac{\mathbf{x}^{\mathbf{p}+\mathbf{n}}}{\mathbf{p}!}
		\frac{(\beta-\beta_1)_{\langle\mathbf{n}\rangle}}{\mathbf{n}!}
		\frac{\mathbf{y}^{\mathbf{j}}}{\mathbf{j}!}\notag\\
		&\hspace{1cm}=\sum_{\mathbf{j},\mathbf{n}=\mathbf{0}}^{\infty}
		\frac{(\alpha)_{\langle\mathbf{n}\rangle+\langle\mathbf{j}\rangle}}{(\gamma)_{\langle\mathbf{n}\rangle+\langle\mathbf{j}\rangle}}
		\frac{\mathbf{y}^{\mathbf{j}}}{\mathbf{j}!}\mathbf{x}^{\mathbf{n}}
		\sum_{\mathbf{p}=\mathbf{0}}^{\mathbf{n}}\frac{(\beta_1)_{\langle\mathbf{p}\rangle}(\beta-\beta_1)_{\langle\mathbf{n}\rangle-\langle\mathbf{p}\rangle}}{(\mathbf{n}-\mathbf{p})!\mathbf{p}!}~~(\mathbf{n}\rightarrow \mathbf{n}-\mathbf{p}).\label{Lemma-Proof-A}
	\end{align}
	
	The inner multiple summation in the last step on the right-hand side of \eqref{Lemma-Proof-A} can be simplified. 
	
	Indeed, by making use of an obvious relationship: $(1-\langle \mathbf{z}\rangle)^{-\lambda_1-\lambda_2}
	=(1-\langle \mathbf{z}\rangle)^{-\lambda_1}
	(1-\langle \mathbf{z}\rangle)^{-\lambda_2}$, we have
	\[
	\sum_{\mathbf{m}=\mathbf{0}}^{\infty}(\lambda_1+\lambda_2)_{\langle\mathbf{m}\rangle}\frac{\mathbf{z}^{\mathbf{m}}}{\mathbf{m}!}
	=\sum_{\mathbf{p},\mathbf{q}=\mathbf{0}}^{\infty}
	(\lambda_1)_{\langle\mathbf{p}\rangle}
	(\lambda_2)_{\langle\mathbf{q}\rangle}
	\frac{\mathbf{z}^{\mathbf{p}+\mathbf{q}}}{\mathbf{p}!\mathbf{q}!}
	=\sum_{\mathbf{m}=\mathbf{0}}^{\infty}
	\left(\sum_{\mathbf{p}=\mathbf{0}}^{\mathbf{m}}
	\frac{(\lambda_1)_{\langle\mathbf{p}\rangle}
		(\lambda_2)_{\langle\mathbf{m}\rangle-\langle\mathbf{p}\rangle}}{\mathbf{p}!(\mathbf{m}-\mathbf{q})!}\right)\mathbf{z}^{\mathbf{m}},
	\]
	which at once leads to	
	\begin{equation}\label{Lemma-Proof-B}
		\frac{(\lambda_1+\lambda_2)_{\langle\mathbf{m}\rangle}}{\mathbf{m}!}=\sum_{\mathbf{p}=\mathbf{0}}^{\mathbf{m}}
		\frac{(\lambda_1)_{\langle\mathbf{p}\rangle}
			(\lambda_2)_{\langle\mathbf{m}\rangle-\langle\mathbf{p}\rangle}}{\mathbf{p}!(\mathbf{m}-\mathbf{q})!}.
	\end{equation}
	
	It follows therefore from \eqref{Lemma-Proof-A} and \eqref{Lemma-Proof-B} that 
	\begin{align*}
		&\sum_{\mathbf{n}=\mathbf{0}}^{\infty}\frac{(\alpha)_{\langle\mathbf{n}\rangle}}{(\gamma)_{\langle\mathbf{n}\rangle}}
		(-\mathbf{x})^{\mathbf{n}}
		\Phi_1\left[\alpha+\langle\mathbf{n}\rangle,\beta_1;\gamma+\langle\mathbf{n}\rangle;\mathbf{x},\bm{\sigma}\circ\mathbf{y}\right]
		L_{\mathbf{n}}^{(\beta_1-\beta-\langle\mathbf{n}\rangle)}\left((1-\bm{\sigma})\mathbf{x}/\mathbf{y}\right)
		\\
		&\hspace{1cm}=\sum_{\mathbf{n},\mathbf{j}=\mathbf{0}}^{\infty}
		\frac{(\alpha)_{\langle\mathbf{n}\rangle+\langle\mathbf{j}\rangle}(\beta)_{\langle\mathbf{n}\rangle}}{(\gamma)_{\langle\mathbf{n}\rangle+\langle\mathbf{j}\rangle}}
		\frac{\mathbf{x}^{\mathbf{n}}}{\mathbf{n}!}
		\frac{\mathbf{y}^{\mathbf{j}}}{\mathbf{j}!}, 
	\end{align*}
	which by appealing to \eqref{Prop-1-Proof-3} leads to the desired result. 
\end{proof}
\begin{remark}
	By letting $\beta_1=0$, $\bm{\sigma}=\mathbf{0}$,  $\mathbf{x}\rightarrow-\mathbf{u}$ and $\mathbf{y}\rightarrow-\mathbf{u}\circ\mathbf{y}$ in \eqref{Lemma-1}, we obtain \eqref{Prop-1-1} of Proposition \ref{Prop-1}.
\end{remark}

\emph{An alternative proof of Lemma \ref{Lemma}.} Considering the importance of Lemma \ref{Lemma}, we present here another approach based on the integration method. Multiplying both sides of \eqref{Prop-1-Proof-4} by $(1-Xt)^{-\beta_1}\mathrm{e}^{Yt}$ and then integrating with respect to the measure $\mu_{\alpha,\gamma-\alpha}(t)$, we obtain 
\begin{align*}
	&\int_{0}^{1}(1-Xt)^{-\beta_1}\mathrm{e}^{Yt}
	\mathrm{e}^{-t\langle \mathbf{u}\circ \mathbf{x}\rangle}(1+t\langle \mathbf{u}\rangle)^{-\beta}
	\mathrm{d}\mu_{\alpha,\gamma-\alpha}(t)\\
	&\hspace{1cm}=\sum_{\mathbf{n}=\mathbf{0}}^{\infty}
	L_{\mathbf{n}}^{(-\beta-\langle \mathbf{n}\rangle)}(\mathbf{x})
	\mathbf{u}^{\mathbf{n}}
	\int_{0}^{1}(1-Xt)^{-\beta_1}\mathrm{e}^{Yt}t^{\langle\mathbf{n}\rangle}\mathrm{d}\mu_{\alpha,\gamma-\alpha}(t)\\
	&\hspace{1cm}=\sum_{\mathbf{n}=\mathbf{0}}^{\infty}
	\frac{(\alpha)_{\langle\mathbf{n}\rangle}}{(\gamma)_{\langle\mathbf{n}\rangle}}
	L_{\mathbf{n}}^{(-\beta-\langle \mathbf{n}\rangle)}(\mathbf{x})
	\mathbf{u}^{\mathbf{n}}
	\int_{0}^{1}(1-Xt)^{-\beta_1}\mathrm{e}^{Yt}\mathrm{d}\mu_{\alpha+\langle\mathbf{n}\rangle,\gamma-\alpha}(t).
\end{align*}
In view of the integral representation \eqref{Prop-1-Proof-6}, the above formula can be expressed as
\begin{align}\label{Lemma-Proof-1}
	&\int_{0}^{1}(1-Xt)^{-\beta_1}\mathrm{e}^{Yt}
	\mathrm{e}^{-t\langle \mathbf{u}\circ \mathbf{x}\rangle}(1+t\langle \mathbf{u}\rangle)^{-\beta}
	\mathrm{d}\mu_{\alpha,\gamma-\alpha}(t)\notag\\
	&\hspace{1cm}=\sum_{\mathbf{n}=\mathbf{0}}^{\infty}
	\frac{(\alpha)_{\langle\mathbf{n}\rangle}}{(\gamma)_{\langle\mathbf{n}\rangle}}
	L_{\mathbf{n}}^{(-\beta-\langle \mathbf{n}\rangle)}(\mathbf{x})
	\mathbf{u}^{\mathbf{n}}
	\Phi_1\left[\alpha+\langle\mathbf{n}\rangle,\beta_1;\gamma+\langle\mathbf{n}\rangle;X,Y\right].
\end{align}
In order to further simplify the right-hand side of \eqref{Lemma-Proof-1} with the help of  \eqref{Prop-1-Proof-6}, we let $\mathbf{u}\rightarrow -\mathbf{u}$ and set $X=\langle\mathbf{u}\rangle$, $Y=\langle\bm{\sigma}\circ\mathbf{y}\rangle$, we then obtain
\begin{align}\label{Lemma-Proof-2}
	&\int_{0}^{1}(1-\langle\mathbf{u}\rangle t)^{-\beta_1-\beta}
	\mathrm{e}^{\langle\bm{\sigma}\circ\mathbf{y}\rangle t+\langle \mathbf{u}\circ \mathbf{x}\rangle t}
	\mathrm{d}\mu_{\alpha,\gamma-\alpha}(t)\notag\\
	&\hspace{1cm}=\sum_{\mathbf{n}=\mathbf{0}}^{\infty}
	\frac{(\alpha)_{\langle\mathbf{n}\rangle}}{(\gamma)_{\langle\mathbf{n}\rangle}}
	L_{\mathbf{n}}^{(-\beta-\langle \mathbf{n}\rangle)}(\mathbf{x})
	(-\mathbf{u})^{\mathbf{n}}
	\Phi_1\left[\alpha+\langle\mathbf{n}\rangle,\beta_1;\gamma+\langle\mathbf{n}\rangle;\langle\mathbf{u}\rangle,\langle\bm{\sigma}\circ\mathbf{y}\rangle\right].
\end{align}
Letting $\mathbf{x}=(\mathbf{1}-\bm{\sigma})\circ \mathbf{y}/\mathbf{u}$ in \eqref{Lemma-Proof-2} and taking into account that 
$\langle\bm{\sigma}\circ\mathbf{y}\rangle+\langle \mathbf{u}\circ \mathbf{x}\rangle=\langle\mathbf{y}\rangle$, we have
\begin{align}\label{Lemma-Proof-3}
	&\int_{0}^{1}(1-\langle\mathbf{u}\rangle t)^{-\beta_1-\beta}
	\mathrm{e}^{\langle\mathbf{y}\rangle t}
	\mathrm{d}\mu_{\alpha,\gamma-\alpha}(t)\notag\\
	&\hspace{1cm}=\sum_{\mathbf{n}=\mathbf{0}}^{\infty}
	\frac{(\alpha)_{\langle\mathbf{n}\rangle}}{(\gamma)_{\langle\mathbf{n}\rangle}}
	L_{\mathbf{n}}^{(-\beta-\langle \mathbf{n}\rangle)}(\mathbf{x})
	(-\mathbf{u})^{\mathbf{n}}
	\Phi_1\left[\alpha+\langle\mathbf{n}\rangle,\beta_1;\gamma+\langle\mathbf{n}\rangle;\langle\mathbf{u}\rangle,\langle\bm{\sigma}\circ\mathbf{y}\rangle\right].
\end{align}
Now, the formula \eqref{Lemma-1} follows by interpreting the right-hand side of \eqref{Lemma-Proof-3} (in view of \eqref{Prop-1-Proof-6}) as the Humbert function $\Phi_1$, and then letting $\beta\rightarrow\beta-\beta_1$.

The following result is based upon the repeated use of the result \eqref{Lemma-1} and thus constitutes a generalized result which has several applications described below.

\begin{theorem}\label{Th-MainR}
	Let $\beta_r-\beta_{r-1}\in\mathbb{C}\setminus\mathbb{Z}_{\leq 0}$ for $r=1,\cdots,L$; where $\beta_0\equiv\beta$. Also, let $\alpha\in\mathbb{C}$, $\gamma\in\mathbb{C}\setminus\mathbb{Z}_{\leq0}$,  $\bm{\sigma}^{(1)},\cdots,\bm{\sigma}^{(L)}\in\mathbb{C}^{k}$ and $\mathbf{x},\mathbf{y}\in\mathbb{C}^k$ with $|x_1|+\cdots+|x_k|<1$. We have
	\begin{align}\label{Th-MainR-1}
		&\Phi_1\left[\alpha,\beta;\gamma;\langle\mathbf{x}\rangle,\langle\mathbf{y}\rangle\right]=\sum_{\mathbf{n}^{(1)},\cdots,\mathbf{n}^{(L)}=\mathbf{0}}^{\infty}\frac{(\alpha)_{\langle\mathbf{n}^{(1)}\rangle+\cdots+\langle\mathbf{n}^{(L)}\rangle}}{(\gamma)_{\langle\mathbf{n}^{(1)}\rangle+\cdots+\langle\mathbf{n}^{(L)}\rangle}}
		(-\mathbf{x})^{\mathbf{n}^{(1)}+\cdots+\mathbf{n}^{(L)}}\notag\\
		&\hspace{1cm}\cdot 
		L_{\mathbf{n}^{(1)}}^{(\beta_1-\beta-\langle\mathbf{n}^{(1)}\rangle)}\left((1-\bm{\sigma}^{(1)})\circ\mathbf{y}/\mathbf{x}\right)
		\cdots 
		L_{\mathbf{n}^{(L)}}^{(\beta_L-\beta_{L-1}-\langle\mathbf{n}^{(L)}\rangle)}\left((1-\bm{\sigma}^{(L)})\circ\bm{\sigma}^{(1)}\cdots\circ\bm{\sigma}^{(L-1)}\circ\mathbf{y}/\mathbf{x}\right)\notag\\
		&\hspace{1cm}\cdot \Phi_1\left[\alpha+\langle\mathbf{n}^{(1)}\rangle+\cdots+\langle\mathbf{n}^{(L)}\rangle,\beta_L;\gamma+\langle\mathbf{n}^{(1)}\rangle+\cdots+\langle\mathbf{n}^{(L)}\rangle;\langle\mathbf{x}\rangle,\langle\bm{\sigma}^{(L)}\circ\cdots\circ\bm{\sigma}^{(1)}\circ\mathbf{y}\rangle\right].
	\end{align}	
\end{theorem}
\begin{proof}
	The derivation of \eqref{Th-MainR-1} is relatively simple. In fact, by repeatedly applying \eqref{Lemma-1} $L$-times, we obtain
	\begin{align*}
		&\Phi_1\left[\alpha,\beta;\gamma;\langle\mathbf{x}\rangle,\langle\mathbf{y}\rangle\right]\\
		&=\sum_{\mathbf{n}^{(1)}=\mathbf{0}}^{\infty}\frac{(\alpha)_{\langle\mathbf{n}^{(1)}\rangle}}{(\gamma)_{\langle\mathbf{n}^{(1)}\rangle}}
		(-\mathbf{x})^{\mathbf{n}^{(1)}}
		L_{\mathbf{n}^{(1)}}^{(\beta_1-\beta-\langle\mathbf{n}^{(1)}\rangle)}\left((1-\bm{\sigma}^{(1)})\circ\mathbf{y}/\mathbf{x}\right)\\
		&\hspace{1cm}\cdot \Phi_1\left[\alpha+\langle\mathbf{n}^{(1)}\rangle,\beta_1;\gamma+\langle\mathbf{n}^{(1)}\rangle;\langle\mathbf{x}\rangle,\langle\bm{\sigma}^{(1)}\circ\mathbf{y}\rangle\right]\\
		&=\sum_{\mathbf{n}^{(1)}=\mathbf{0}}^{\infty}\frac{(\alpha)_{\langle\mathbf{n}^{(1)}\rangle}}{(\gamma)_{\langle\mathbf{n}^{(1)}\rangle}}
		(-\mathbf{x})^{\mathbf{n}^{(1)}}
		L_{\mathbf{n}^{(1)}}^{(\beta_1-\beta-\langle\mathbf{n}^{(1)}\rangle)}\left((1-\bm{\sigma}^{(1)})\circ\mathbf{y}/\mathbf{x}\right)\\
		&\hspace{1cm}\cdot \sum_{\mathbf{n}^{(2)}=\mathbf{0}}^{\infty}\frac{(\alpha+\langle\mathbf{n}^{(1)}\rangle)_{\langle\mathbf{n}^{(2)}\rangle}}{(\gamma+\langle\mathbf{n}^{(1)}\rangle)_{\langle\mathbf{n}^{(2)}\rangle}}
		(-\mathbf{x})^{\mathbf{n}^{(2)}}
		L_{\mathbf{n}^{(2)}}^{(\beta_2-\beta_1-\langle\mathbf{n}^{(2)}\rangle)}\left((1-\bm{\sigma}^{(2)})\circ\bm{\sigma}^{(1)}\circ\mathbf{y}/\mathbf{x}\right)\\
		&\hspace{1cm}\cdot \Phi_1\left[\alpha+\langle\mathbf{n}^{(1)}\rangle+\langle\mathbf{n}^{(2)}\rangle,\beta_2;\gamma+\langle\mathbf{n}^{(1)}\rangle+\langle\mathbf{n}^{(2)}\rangle;\langle\mathbf{x}\rangle,\langle\bm{\sigma}^{(2)}\circ\bm{\sigma}^{(1)}\circ\mathbf{y}\rangle\right]\\
		&=\sum_{\mathbf{n}^{(1)},\mathbf{n}^{(2)}=\mathbf{0}}^{\infty}\frac{(\alpha)_{\langle\mathbf{n}^{(1)}\rangle+\langle\mathbf{n}^{(2)}\rangle}}{(\gamma)_{\langle\mathbf{n}^{(1)}\rangle+\langle\mathbf{n}^{(2)}\rangle}}
		(-\mathbf{x})^{\mathbf{n}^{(1)}+\mathbf{n}^{(2)}}
		L_{\mathbf{n}^{(1)}}^{(\beta_1-\beta-\langle\mathbf{n}^{(1)}\rangle)}\left((1-\bm{\sigma}^{(1)})\circ\mathbf{y}/\mathbf{x}\right)\\
		&\hspace{1cm}\cdot 
		L_{\mathbf{n}^{(2)}}^{(\beta_2-\beta_1-\langle\mathbf{n}^{(2)}\rangle)}\left((1-\bm{\sigma}^{(2)})\circ\bm{\sigma}^{(1)}\circ\mathbf{y}/\mathbf{x}\right)\\
		&\hspace{1cm}\cdot \Phi_1\left[\alpha+\langle\mathbf{n}^{(1)}\rangle+\langle\mathbf{n}^{(2)}\rangle,\beta_2;\gamma+\langle\mathbf{n}^{(1)}\rangle+\langle\mathbf{n}^{(2)}\rangle;\langle\mathbf{x}\rangle,\langle\bm{\sigma}^{(2)}\circ\bm{\sigma}^{(1)}\circ\mathbf{y}\rangle\right]\\
		&=\hspace{0.5cm}\cdots\\
		&=\sum_{\mathbf{n}^{(1)},\cdots,\mathbf{n}^{(L)}=\mathbf{0}}^{\infty}\frac{(\alpha)_{\langle\mathbf{n}^{(1)}\rangle+\cdots+\langle\mathbf{n}^{(L)}\rangle}}{(\gamma)_{\langle\mathbf{n}^{(1)}\rangle+\cdots+\langle\mathbf{n}^{(L)}\rangle}}
		(-\mathbf{x})^{\mathbf{n}^{(1)}+\cdots+\mathbf{n}^{(L)}}
		L_{\mathbf{n}^{(1)}}^{(\beta_1-\beta-\langle\mathbf{n}^{(1)}\rangle)}\left((1-\bm{\sigma}^{(1)})\circ\mathbf{y}/\mathbf{x}\right)\\
		&\hspace{1cm}\cdot 
		L_{\mathbf{n}^{(2)}}^{(\beta_2-\beta_1-\langle\mathbf{n}^{(2)}\rangle)}\left((1-\bm{\sigma}^{(2)})\circ\bm{\sigma}^{(1)}\circ\mathbf{y}/\mathbf{x}\right)\\
		&\hspace{1cm}\cdots 
		L_{\mathbf{n}^{(L)}}^{(\beta_L-\beta_{L-1}-\langle\mathbf{n}^{(L)}\rangle)}\left((1-\bm{\sigma}^{(L)})\circ\bm{\sigma}^{(1)}\cdots\circ\bm{\sigma}^{(L-1)}\circ\mathbf{y}/\mathbf{x}\right)\\
		&\hspace{1cm}\cdot \Phi_1\left[\alpha+\langle\mathbf{n}^{(1)}\rangle+\cdots+\langle\mathbf{n}^{(L)}\rangle,\beta_L;\gamma+\langle\mathbf{n}^{(1)}\rangle+\cdots+\langle\mathbf{n}^{(L)}\rangle;\langle\mathbf{x}\rangle,\langle\bm{\sigma}^{(L)}\circ\cdots\circ\bm{\sigma}^{(1)}\circ\mathbf{y}\rangle\right].
	\end{align*} 
	The desired Theorem \ref{Th-MainR} then follows immediately by the implementation of the principle of mathematical induction.
\end{proof}

\subsection{Special cases}

As our main result, Theorem \ref{Th-MainR} has many useful consequences. In this subsection, we will discuss these specific cases in detail.

Letting $\mathbf{x}=-\mathbf{u}$, $\mathbf{y}=-\mathbf{u}\circ(\mathbf{w}^{(1)}+\cdots+\mathbf{w}^{(L+1)})$ and $\bm{\sigma}^{(i)}=\mathbf{1}-\mathbf{w}^{(i)}/(\mathbf{w}^{(i)}+\cdots+\mathbf{w}^{(L+1)})$ in Theorem \ref{Th-MainR}, we obtain the following corollary.

\begin{corollary}\label{Cor-1}
	Let $\beta_r-\beta_{r-1}\in\mathbb{C}\setminus\mathbb{Z}_{\leq 0}$ for $r=1,\cdots,L$, where $\beta_0\equiv\beta$. Also, let $\alpha\in\mathbb{C}$, $\gamma\in\mathbb{C}\setminus\mathbb{Z}_{\leq0}$,  $\mathbf{w}^{(1)},\cdots,\mathbf{w}^{(L+1)}\in\mathbb{C}^{k}$ and $\mathbf{u}\in\mathbb{C}^k$ with $|u_1|+\cdots+|u_k|<1$. We have
	\begin{align*}
		&\Phi_1\left[\alpha,\beta;\gamma;-\langle\mathbf{u}\rangle,-\langle\mathbf{u}\circ(\mathbf{w}^{(1)}+\cdots+\mathbf{w}^{(L+1)})\rangle\right]=\sum_{\mathbf{n}^{(1)},\cdots,\mathbf{n}^{(L)}=\mathbf{0}}^{\infty}\frac{(\alpha)_{\langle\mathbf{n}^{(1)}\rangle+\cdots+\langle\mathbf{n}^{(L)}\rangle}}{(\gamma)_{\langle\mathbf{n}^{(1)}\rangle+\cdots+\langle\mathbf{n}^{(L)}\rangle}}
		\mathbf{u}^{\mathbf{n}^{(1)}+\cdots+\mathbf{n}^{(L)}}\notag\\
		&\hspace{1cm}\cdot 
		L_{\mathbf{n}^{(1)}}^{(\beta_1-\beta-\langle\mathbf{n}^{(1)}\rangle)}(\mathbf{w}^{(1)})
		\cdots 
		L_{\mathbf{n}^{(L)}}^{(\beta_L-\beta_{L-1}-\langle\mathbf{n}^{(L)}\rangle)}(\mathbf{w}^{(L)})\notag\\
		&\hspace{1cm}\cdot \Phi_1\left[\alpha+\langle\mathbf{n}^{(1)}\rangle+\cdots+\langle\mathbf{n}^{(L)}\rangle,\beta_L;\gamma+\langle\mathbf{n}^{(1)}\rangle+\cdots+\langle\mathbf{n}^{(L)}\rangle;-\langle\mathbf{u}\rangle,-\langle\mathbf{u}\circ\mathbf{w}^{(L+1)}\rangle\right].
	\end{align*}
\end{corollary}

Corollary \ref{Cor-1} generalizes a result of Tremblay and Lavertu \cite[p. 14, Eq. (2.5)]{Tremblay-Lavertu-1972}.

Letting further $\beta_L=0$ and $\mathbf{w}^{(L+1)}=\mathbf{0}=(0,\cdots,0)$ in Corollary \ref{Cor-1}, we then get following result. 
\begin{corollary}\label{Cor-2}
	Let $-\beta_{L-1},\beta_r-\beta_{r-1}\in\mathbb{C}\setminus\mathbb{Z}_{\leq 0}$ for $r=1,\cdots,L-1$, where $\beta_0\equiv\beta$. Also, let $\alpha\in\mathbb{C}$, $\gamma\in\mathbb{C}\setminus\mathbb{Z}_{\leq0}$,  $\mathbf{w}^{(1)},\cdots,\mathbf{w}^{(L)}\in\mathbb{C}^{k}$ and $\mathbf{u}\in\mathbb{C}^k$ with $|u_1|+\cdots+|u_k|<1$. We have
	\begin{align*}
		\Phi_1\left[\alpha,\beta;\gamma;-\langle\mathbf{u}\rangle,-\langle\mathbf{u}\circ(\mathbf{w}^{(1)}+\cdots+\mathbf{w}^{(L)})\rangle\right]
		&=\sum_{\mathbf{n}^{(1)},\cdots,\mathbf{n}^{(L)}=\mathbf{0}}^{\infty}\frac{(\alpha)_{\langle\mathbf{n}^{(1)}\rangle+\cdots+\langle\mathbf{n}^{(L)}\rangle}}{(\gamma)_{\langle\mathbf{n}^{(1)}\rangle+\cdots+\langle\mathbf{n}^{(L)}\rangle}}
		\mathbf{u}^{\mathbf{n}^{(1)}+\cdots+\mathbf{n}^{(L)}}\notag\\
		&\hspace{0.5cm}\cdot 
		L_{\mathbf{n}^{(L)}}^{(-\beta_{L-1}-\langle\mathbf{n}^{(L)}\rangle)}(\mathbf{w}^{(L)})\prod_{j=1}^{L-1}
		L_{\mathbf{n}^{(j)}}^{(\beta_j-\beta_{j-1}-\langle\mathbf{n}^{(j)}\rangle)}(\mathbf{w}^{(j)}).
	\end{align*}
\end{corollary}

Corollary \ref{Cor-2} provides a generalization of formula \eqref{GF-TremblayLavertu}. The derivation of the next three corollaries will rely on some known hypergeometric identities.

\begin{corollary}\label{Cor-3}
	Let $a_r\in\mathbb{C}\setminus\mathbb{Z}_{\leq 0}$ $(r=1,\cdots,L+1)$ and $a_1+\cdots+a_{L+1}\in\mathbb{C}\setminus\mathbb{Z}_{\leq0}$. Also, let  $\mathbf{w}^{(1)},\cdots,\mathbf{w}^{(L+1)}\in\mathbb{C}^{k}$ and $\mathbf{u}\in\mathbb{C}^k$ with $\langle\mathbf{u}\rangle=-1$. We have
	\begin{align}\label{Cor-3-1}
		&L_m^{(a_1+\cdots+a_{L+1})}\left(-\langle\mathbf{u}\circ(\mathbf{w}^{(1)}+\cdots+\mathbf{w}^{(L+1)})\rangle\right)
		=\sum_{\mathbf{n}^{(1)},\cdots,\mathbf{n}^{(L)}
			=\mathbf{0}}^{\langle\mathbf{n}^{(1)}\rangle+\cdots+\langle\mathbf{n}^{(L)}\rangle\leq m}
		(-\mathbf{u})^{\mathbf{n}^{(1)}+\cdots+\mathbf{n}^{(L)}}
		\notag\\
		&\hspace{1cm}\cdot 
		L_{\mathbf{n}^{(1)}}^{(a_1-\langle\mathbf{n}^{(1)}\rangle)}(\mathbf{w}^{(1)})
		\cdots 
		L_{\mathbf{n}^{(L)}}^{(a_L-\langle\mathbf{n}^{(L)}\rangle)}(\mathbf{w}^{(L)})
		L_{m-\langle\mathbf{n}^{(1)}\rangle-\cdots-\langle\mathbf{n}^{(L)}\rangle}^{(a_{L+1}+\langle\mathbf{n}^{(1)}\rangle+\cdots+\langle\mathbf{n}^{(L)}\rangle)}
		(-\langle\mathbf{u}\circ\mathbf{w}^{(L+1)}\rangle).
	\end{align}
\end{corollary}
\begin{proof}
	Letting $\langle\mathbf{u}\rangle=-1$ in Corollary \ref{Cor-1} and taking into account the formula (\cite[p. 15, Eq. (3.5)]{Tremblay-Lavertu-1972})
	\[
	\Phi_1[a,b;c;1,y]=
	\frac{\Gamma(c)\Gamma(c-a-b)}{\Gamma(c-a)\Gamma(c-b)}
	{}_{1}F_{1}\left[\begin{matrix}
		a\\
		c-b
	\end{matrix};y\right],\quad \Re(c-a-b)>0,
	\]
	we obtain 
	\begin{align}\label{Cor-3-Proof-1}
		&\frac{\Gamma(\gamma-\beta_L)\Gamma(\gamma-\alpha-\beta)}{\Gamma(\gamma-\alpha-\beta_L)\Gamma(\gamma-\beta)}
		{}_{1}F_{1}\left[\begin{matrix}
			\alpha\\
			\gamma-\beta
		\end{matrix};-\langle\mathbf{u}\circ(\mathbf{w}^{(1)}+\cdots+\mathbf{w}^{(L+1)})\rangle\right]\notag\\
		&\hspace{1cm}=\sum_{\mathbf{n}^{(1)},\cdots,\mathbf{n}^{(L)}=\mathbf{0}}^{\infty}\frac{(\alpha)_{\langle\mathbf{n}^{(1)}\rangle+\cdots+\langle\mathbf{n}^{(L)}\rangle}}{(\gamma-\beta_L)_{\langle\mathbf{n}^{(1)}\rangle+\cdots+\langle\mathbf{n}^{(L)}\rangle}}
		\mathbf{u}^{\mathbf{n}^{(1)}+\cdots+\mathbf{n}^{(L)}}
		L_{\mathbf{n}^{(1)}}^{(\beta_1-\beta-\langle\mathbf{n}^{(1)}\rangle)}(\mathbf{w}^{(1)})
		\cdots 
		L_{\mathbf{n}^{(L)}}^{(\beta_L-\beta_{L-1}-\langle\mathbf{n}^{(L)}\rangle)}(\mathbf{w}^{(L)})\notag\\
		&\hspace{1.5cm}\cdot
		{}_{1}F_{1}\left[\begin{matrix}
			\alpha+\langle\mathbf{n}^{(1)}\rangle+\cdots+\langle\mathbf{n}^{(L)}\rangle\\
			\gamma-\beta_L+\langle\mathbf{n}^{(1)}\rangle+\cdots+\langle\mathbf{n}^{(L)}\rangle
		\end{matrix};-\langle\mathbf{u}\circ\mathbf{w}^{(L+1)}\rangle\right],
	\end{align}
	where $\Re(\gamma-\beta-\alpha)>0$ and $\Re(\gamma-\beta_L-\alpha)>0$. 
	
	Next, let $\alpha=-m$ $(m\in\mathbb{Z}_{\geq0})$, $\beta_r-\beta_{r-1}=a_r\in\mathbb{C}\setminus\mathbb{Z}_{\leq0}$  $(r=1,\cdots,L; ~\beta_0\equiv\beta)$ and 
	\[
	\gamma-\beta=a_1+\cdots+a_L+a_{L+1}+1,
	\]  
	in \eqref{Cor-3-Proof-1}. Then $\gamma-\beta_L=a_{L+1}+1$ and thus \eqref{Cor-3-Proof-1} becomes
	\begin{align}\label{Cor-3-Proof-2}
		&\frac{(a_1+\cdots+a_{L+1}+1)_m}{(a_{L+1}+1)_m}
		{}_{1}F_{1}\left[\begin{matrix}
			-m\\
			a_1+\cdots+a_{L+1}+1
		\end{matrix};-\langle\mathbf{u}\circ(\mathbf{w}^{(1)}+\cdots+\mathbf{w}^{(L+1)})\rangle\right]\notag\\
		&\hspace{1cm}=\sum_{\mathbf{n}^{(1)},\cdots,\mathbf{n}^{(L)}
			=\mathbf{0}}^{\langle\mathbf{n}^{(1)}\rangle+\cdots+\langle\mathbf{n}^{(L)}\rangle\leq m}\frac{(-m)_{\langle\mathbf{n}^{(1)}\rangle+\cdots+\langle\mathbf{n}^{(L)}\rangle}}{(a_{L+1}+1)_{\langle\mathbf{n}^{(1)}\rangle+\cdots+\langle\mathbf{n}^{(L)}\rangle}}
		\mathbf{u}^{\mathbf{n}^{(1)}+\cdots+\mathbf{n}^{(L)}}
		L_{\mathbf{n}^{(1)}}^{(a_1-\langle\mathbf{n}^{(1)}\rangle)}(\mathbf{w}^{(1)})
		\cdots 
		L_{\mathbf{n}^{(L)}}^{(a_L-\langle\mathbf{n}^{(L)}\rangle)}(\mathbf{w}^{(L)})\notag\\
		&\hspace{1.5cm}\cdot
		{}_{1}F_{1}\left[\begin{matrix}
			-m+\langle\mathbf{n}^{(1)}\rangle+\cdots+\langle\mathbf{n}^{(L)}\rangle\\
			a_{L+1}+1+\langle\mathbf{n}^{(1)}\rangle+\cdots+\langle\mathbf{n}^{(L)}\rangle
		\end{matrix};-\langle\mathbf{u}\circ\mathbf{w}^{(L+1)}\rangle\right].
	\end{align}
	
	Finally, interpreting the terminated ${}_{1}F_{1}$ functions in \eqref{Cor-3-Proof-2} as the Laguerre polynomial via \eqref{Def-LaguerreP} yields the desired result \eqref{Cor-3-1}.
\end{proof}
\begin{remark}
	When $k=1$ in Corallary \ref{Cor-3}, we obtain the following Tremblay and Lavertu's result \cite[p. 16, Eq. (3.9)]{Tremblay-Lavertu-1972}:
	\begin{equation}\label{Cor-3-2}
		L_m^{(a_1+\cdots+a_{L+1})}\left(\sum_{j=1}^{L+1}w_j\right)
		=\sum_{n_1,\cdots,n_L=0}^{n_1+\cdots+n_L\leq m}
		L_{n_1}^{(a_1-n_1)}(w_1)
		\cdots 
		L_{n_L}^{(a_L-n_L)}(w_L)
		L_{m-n_1-\cdots-n_L}^{(a_{L+1}+n_1+\cdots+n_L)}
		(w_{L+1}).
	\end{equation}
	The formula \eqref{Cor-3-2} is a remarkable result in the sense that it can be regarded as a generalization of the usual addition theorem (\cite[p. 460, Eq. (18.18.10)]{NIST Handbook}) for the Laguerre polynomials \eqref{Def-LaguerreP}. In fact, our Corollary \ref{Cor-3} is quite an interesting result since the left-hand side of \eqref{Cor-3-1} consists of a Laguerre polynomial of sum of finite number of variables while the right-hand side of \eqref{Cor-3-1} involves the product of multivariate Laguerre polynomials.
\end{remark}

\begin{corollary}\label{Cor-4}
	Let $\mathbf{u}, \mathbf{w}^{(1)}, \cdots, \mathbf{w}^{(L)} \in \mathbb{C}^k$ with $\langle \mathbf{u} \rangle =1$. Also, let $\Re(\beta_L)<1$ and $\alpha-\beta_L+1 \in \mathbb{C}\setminus\mathbb{Z}_{\leq 0}$. We have
	\begin{align}\label{Cor-4-1}
		&\Phi_1\left[\alpha,\beta;\alpha-\beta_L+1;-1, -\langle \mathbf{u} \circ (\mathbf{w}^{(1)}+\cdots+\mathbf{w}^{(L)}) \rangle\right]
		=\frac{\Gamma(\alpha-\beta_L+1)}{\Gamma(\alpha)}\notag\\
		&\hspace{1cm} \cdot\sum_{\mathbf{n}^{(1)},\cdots,\mathbf{n}^{(L)}=\mathbf{0}}^{\infty} \frac{\Gamma\left(\frac{1}{2}(\alpha+\langle\mathbf{n}^{(1)}\rangle +\cdots+ \langle\mathbf{n}^{(L)}\rangle)+1\right)~\mathbf{u}^{\mathbf{n}^{(1)}+\cdots+\mathbf{n}^{(L)}} }{\Gamma\left(\frac{1}{2}(\alpha+\langle\mathbf{n}^{(1)}\rangle +\cdots+ \langle\mathbf{n}^{(L)}\rangle)-\beta_L+1\right)~ (\alpha+\langle\mathbf{n}^{(1)}\rangle +\cdots+ \langle\mathbf{n}^{(L)}\rangle)}\notag\\
		&\hspace{1cm}\cdot 
		L_{\mathbf{n}^{(1)}}^{(\beta_1-\beta-\langle \mathbf{n}^{(1)} \rangle)}(\mathbf{w}^{(1)}) \cdots L_{\mathbf{n}^{(L)}}^{(\beta_L-\beta_{L-1}-\langle \mathbf{n}^{(L)} \rangle)}(\mathbf{w}^{(L)}).
	\end{align}
\end{corollary}
\begin{proof}
	Letting $\langle\mathbf{u}\rangle=1, \mathbf{w}^{(L+1)}=\mathbf{0}=(0,\dots,0)$ and $\gamma=\alpha-\beta_L+1$ in Corollary \ref{Cor-1}, we have
	\begin{align}\label{Cor-4-Proof-1}
		&\Phi_1\left[\alpha,\beta;\alpha-\beta_L+1;-1, -\langle \mathbf{u} \circ (\mathbf{w}^{(1)}+\cdots+\mathbf{w}^{(L)}) \rangle\right] \notag\\
		&\hspace{1cm}=\sum_{\mathbf{n}^{(1)},\cdots,\mathbf{n}^{(L)}=\mathbf{0}}^{\infty} \frac{(\alpha)_{\langle \mathbf{n}^{(1)} \rangle+\cdots+\langle \mathbf{n}^{(L)} \rangle}}{(\alpha-\beta_L+1)_{\langle \mathbf{n}^{(1)} \rangle+\cdots+\langle \mathbf{n}^{(L)} \rangle}} \mathbf{u}^{\mathbf{n}^{(1)}+\cdots+\mathbf{n}^{(L)}} L_{\mathbf{n}^{(1)}}^{(\beta_1-\beta-\langle \mathbf{n}^{(1)} \rangle)}(\mathbf{w}^{(1)})\notag \\
		&\hspace{1.5cm} \cdots L_{\mathbf{n}^{(L)}}^{(\beta_L-\beta_{L-1}-\langle \mathbf{n}^{(L)} \rangle)}(\mathbf{w}^{(L)})  {}_2F_1 \left[ \begin{matrix}
			\alpha+\langle \mathbf{n}^{(1)} \rangle+\cdots+\langle \mathbf{n}^{(L)} \rangle,\beta_L \\
			\alpha-\beta_L+\langle \mathbf{n}^{(1)} \rangle+\cdots+\langle \mathbf{n}^{(L)} \rangle+1
		\end{matrix}; -1 \right] .
	\end{align}
	The hypergeometric sum in \eqref{Cor-4-Proof-1} can be evaluated with the help of Kummer's summation theorem \cite[p. 68, Eq. (2)]{Rainville-Book-1960}:
	\[
	{}_2F_1
	\left[ \begin{matrix}
		a, b\\
		a-b+1
	\end{matrix}; -1\right]
	=\frac{\Gamma(a-b+1) \Gamma(\frac{a}{2}+1)}{\Gamma(\frac{a}{2}-b+1)\Gamma(a+1)},
	~\Re(b)<1,~1+a-b\in\mathbb{C}\setminus\mathbb{Z}_{\leq 0}. 
	\]
	After a little simplification, we obtain \eqref{Cor-4-1}.
\end{proof}
\begin{remark}
	If $k=1$ in Corollary \ref{Cor-4}, then it follows that $u=1$ and thus \eqref{Cor-4-1} reduces to Tremblay and Lavertu's result \cite[p. 15, Eq. (3.4)]{Tremblay-Lavertu-1972}.
\end{remark}

\begin{corollary}\label{Cor-5}
	Let $\mathbf{u}, \mathbf{w}^{(1)}, \cdots, \mathbf{w}^{(L)} \in \mathbb{C}^k$ with $\langle \mathbf{u} \rangle =1$. Also, let $\Re(\beta)<1$ and $\alpha-\beta+1 \in \mathbb{C}\setminus\mathbb{Z}_{\leq 0}$. We have
	\begin{align*}
		&\frac{\Gamma(\frac{\alpha}{2})}{2\Gamma(\frac{\alpha}{2}-\beta+1)} {}_1F_2 \left[\begin{matrix}
			\frac{\alpha}{2} \\
			\frac{1}{2}, \frac{\alpha}{2}-\beta+1
		\end{matrix}; \frac{\langle \mathbf{u} \circ (\mathbf{w}^{(1)}+\cdots+\mathbf{w}^{(L)}) \rangle^2}{4} \right]\notag\\
		&\hspace{1cm}- \frac{\Gamma(\frac{\alpha}{2}+\frac{1}{2})}{2\Gamma(\frac{\alpha}{2}-\beta+\frac{3}{2})} \langle \mathbf{u} \circ (\mathbf{w}^{(1)}+\cdots+\mathbf{w}^{(L)}) \rangle \cdot {}_1F_2\left[ \begin{matrix}
			\frac{\alpha}{2}+\frac{1}{2} \\
			\frac{3}{2}, \frac{\alpha}{2}-\beta+\frac{3}{2}
		\end{matrix}; \frac{\langle \mathbf{u} \circ (\mathbf{w}^{(1)}+\cdots+\mathbf{w}^{(L)}) \rangle^2}{4} \right]\notag\\
		&=\sum_{\mathbf{n}^{(1)},\cdots,\mathbf{n}^{(L)}=\mathbf{0}}^{\infty} \frac{\Gamma\left(\frac{1}{2}(\alpha+\langle\mathbf{n}^{(1)}\rangle +\cdots+ \langle\mathbf{n}^{(L)}\rangle)+1\right)~\mathbf{u}^{\mathbf{n}^{(1)}+\cdots+\mathbf{n}^{(L)}} }{\Gamma\left(\frac{1}{2}(\alpha+\langle\mathbf{n}^{(1)}\rangle +\cdots+ \langle\mathbf{n}^{(L)}\rangle)-\beta+1\right)~ (\alpha+\langle\mathbf{n}^{(1)}\rangle +\cdots+ \langle\mathbf{n}^{(L)}\rangle)} \notag\\
		&\hspace{1cm}\cdot L_{\mathbf{n}^{(1)}}^{(\beta_1-\beta-\langle \mathbf{n}^{(1)} \rangle)}(\mathbf{w}^{(1)}) \cdots L_{\mathbf{n}^{(L-1)}}^{(\beta_{L-1}-\beta_{L-2}-\langle \mathbf{n}^{(L-1)} \rangle)}(\mathbf{w}^{(L-1)}) L_{\mathbf{n}^{(L)}}^{(\beta-\beta_{L-1}-\langle \mathbf{n}^{(L)} \rangle)}(\mathbf{w}^{(L)}) .
	\end{align*}
\end{corollary}
\begin{proof}
	The result follows by letting $\beta_L=\beta$ in Corollary \ref{Cor-4} and then making use of the formula \cite[p. 4, Eq. (2.9)]{Hang-Hu-Luo-2026}:
	\begin{align*}
		&\Phi_1\left[a,b;a-b+1;-1,y\right]
		=\frac{\Gamma(a-b+1)}{2\,\Gamma(a)}\Bigg\{
		\frac{\Gamma(\frac{a}{2})}{\Gamma(\frac{a}{2}-b+1)}
		{}_1F_2\left[\begin{matrix}
			\frac{a}{2}\\
			\frac{1}{2},\frac{a}{2}-b+1
		\end{matrix};\frac{y^2}{4}\right]\\
		&\hspace{1cm}+\frac{\Gamma(\frac{a}{2}+\frac{1}{2})}{\Gamma(\frac{a}{2}-b+\frac{3}{2})}\,y\cdot{}_1F_2\left[\begin{matrix}
			\frac{a}{2}+\frac{1}{2}\\
			\frac{3}{2},\frac{a}{2}-b+\frac{3}{2}
		\end{matrix};\frac{y^2}{4}\right]
		\Bigg\}, \quad \Re(b)<1,~a-b+1\in\mathbb{C}\setminus\mathbb{Z}_{\geq0}.	
	\end{align*}
\end{proof}
\begin{remark}
	In particular, when $k=1$ in Corollary \ref{Cor-5}, we obtain the following interesting result: 
	\begin{align*}
		&\frac{\Gamma(\frac{\alpha}{2})}{2\Gamma(\frac{\alpha}{2}-\beta+1)} {}_1F_2 \left[\begin{matrix}
			\frac{\alpha}{2} \\
			\frac{1}{2}, \frac{\alpha}{2}-\beta+1
		\end{matrix}; \frac{(w_1+\cdots+w_L)^2}{4} \right]\notag\\
		&\hspace{1cm}- \frac{\Gamma(\frac{\alpha}{2}+\frac{1}{2})}{2\Gamma(\frac{\alpha}{2}-\beta+\frac{3}{2})} (w_1+\cdots+w_L) \cdot {}_1F_2\left[ \begin{matrix}
			\frac{\alpha}{2}+\frac{1}{2} \\
			\frac{3}{2}, \frac{\alpha}{2}-\beta+\frac{3}{2}
		\end{matrix}; \frac{(w_1+\cdots+w_L)^2}{4} \right]\notag\\
		&=\sum_{n_1,\cdots,n_L=0}^{\infty} \frac{\Gamma\left(\frac{1}{2}(\alpha+n_1+\cdots+n_L)+1\right) }{\Gamma\left(\frac{1}{2}(\alpha+n_1+\cdots+n_L)-\beta+1\right) (\alpha+n_1+\cdots+n_L)} \notag\\
		&\hspace{1cm}\cdot L_{n_1}^{(\beta_1-\beta-n_1)}(w_1) \cdots L_{n_{L-1}}^{(\beta_{L-1}-\beta_{L-2}-n_{L-1})}(w_{L-1}) L_{n_L}^{(\beta-\beta_{L-1}-n_L)}(w_L) .
	\end{align*}
\end{remark}

\section{Product formulas for $L_{\mathbf{n}}^{(\alpha)}(\mathbf{x})$}\label{Sec-ProductFormulas}

The representation of the product of two special functions as an integral in terms of functions of the same class has long been a subject of considerable interest, and even to this day, research along this direction continues to produce significant results (see, for example, \cite{Cohl-Durand-2026}). There are several important product formulas for Laguerre polynomials $L_n^{(\alpha)}(x)$, which have useful applications in many branches of mathematics. Among all these formulas, perhaps the most well-known is \emph{Watson's formula} (see \cite{Watson-1939}; see also \cite{Koornwinder-1977} and \cite{Stempak-1988})
\begin{align}\label{WatsonIntegral-1939}
	L_n^{(\alpha)}(x^2)L_n^{(\alpha)}(y^2)
	&=\frac{2^{\alpha}\Gamma(\alpha+1+n)}{\sqrt{2\pi}~n!}\int_0^\pi L_n^{(\alpha)}(x^2+y^2+2xy\cos\theta)\notag\\
	&\hspace{2cm}\cdot \mathrm{e}^{-xy\cos\theta}j_{\alpha-1/2}(xy\sin\theta)\sin^{2\alpha}\theta\mathrm{d}\theta,
\end{align}
where $\alpha>-1/2$, $x,y>0$, and $j_\alpha(x):=x^{-\alpha}J_{\alpha}(x)$ with $J_{\alpha}(x)$ being the Bessel function. Note that \eqref{WatsonIntegral-1939} can also be written as a double integral. In another article of Watson \cite{Watson-1938}, he derived 
\begin{align}\label{WatsonIntegral-1938}
	\frac{L_m^{(\alpha)}(x)}{\Gamma(\alpha+m+1)}
	\frac{L_n^{(\alpha)}(x)}{\Gamma(\alpha+n+1)}
	&=\frac{2^{m+n+1}}{\pi \Gamma(\alpha+\frac{1}{2})\Gamma(\frac{1}{2})}
	\int_0^{\frac{\pi}{2}}
	\int_0^{\pi}\sin^{2\alpha}\theta\cos^{m+n}\phi\cos(m-n)\phi\notag\\
	&\hspace{2cm}\cdot\frac{L_{m+n}^{(\alpha)}(x+x\sec\phi\cos\theta)}{\Gamma(\alpha+m+n+1)}\mathrm{d}\theta\mathrm{d}\phi,
\end{align} 
which includes Bailey's formula for the Hermite polynomials. Later, Carlitz \cite{Carlitz-1962} obtained a slightly different formula, namely the special case of \eqref{ProductFormula-MLaguerre} with $k=1$. Carlitz's technique is quite concise and effective, and has therefore found wide applications in Bessel polynomials \cite{Chatterjea-1963}, Jacobi polynomials \cite{Chatterjea-1964}, Lauricella polynomials \cite{Liu-Lin-Lu-Srivastava-2013a}, Rice polynomials \cite{Manocha-1969} etc.

Since there is no direct proof of \eqref{ProductFormula-MLaguerre} available in the literature, we provide here a direct proof of a more general version of the result in terms of contour integrals. 

We need two auxiliary integral formulas. Let $\mathbb{T}$ denote the positively oriented unit circle. First, recall that
\begin{equation}\label{AuxiliaryIF-1}
	\frac{(m+n)!}{m!\, n!}=\binom{m+n}{m}
	=\frac{1}{2\pi\mathrm{i}}\int_{\mathbb{T}}(1+z)^{m+n}z^{-m}
	\frac{\mathrm{d}z}{z},
\end{equation}
where $m,n\in\mathbb{Z}_{\geq0}$. This is an easy identity which can be proved by expanding $(1+z)^{m+n}$ via the binomial theorem and then evaluating the integral of the form $\int_{\mathbb{T}}z^{k}\mathrm{d}z$ explicitly (see, for example, \cite[p. 128, Lemma 1]{Chakraborty-Roy-2024}). Next, we require the integral \cite[p. 93, Theorem 1]{Stade-1994}: 
\begin{equation}\label{AuxiliaryIF-2}
	\frac{\Gamma(x+y+1)}{\Gamma(x+1)\Gamma(y+1)}=\frac{1}{2\pi\mathrm{i}}\int_{\mathbb{T}}\left(1+\frac{1}{u}\right)^{x}(1+u)^{y}\frac{\mathrm{d}u}{u},
\end{equation}
where $\Re(x+y)>-1$.

\begin{theorem}\label{ProductF-Th}
	Let $\Re(\alpha+\beta)>-1$. Then we have
	\begin{align}\label{ProductF-Th-1}
		\frac{L_{\mathbf{m}}^{(\alpha)}(\mathbf{x})}{\Gamma(\alpha+1+\langle\mathbf{m}\rangle)} \frac{L_{\mathbf{n}}^{(\beta)}(\mathbf{y})}{\Gamma(\beta+1+\langle\mathbf{n}\rangle)}
		&=
		\frac{1}{(2\pi\mathrm{i})^{k+1}}\int_{\mathbb{T}^{k+1}}
		\omega_{\alpha,\beta}(u)
		\prod_{r=1}^{k}\omega_{m_r,n_r}(z_r)\notag\\
		&\hspace{1.5cm}\cdot
		\frac{L_{\mathbf{m}+\mathbf{n}}^{(\alpha+\beta)}(\widetilde{\Xi}(\mathbf{x},\mathbf{y};u, \mathbf{z}))}{\Gamma(\alpha+\beta+\langle\mathbf{m}+\mathbf{n}\rangle+1)}\mathrm{d}u\mathrm{d}\mathbf{z},
	\end{align}
	where $\widetilde{\Xi}(\mathbf{x},\mathbf{y};u, \mathbf{z})=(\widetilde{\xi}_1+\widetilde{\eta}_1,\cdots,\widetilde{\xi}_k+\widetilde{\eta}_k)$ with
	\begin{align}
		\widetilde{\xi}_r \equiv \widetilde{\xi}_r(u,z_r,x_r)&:=\frac{1+u}{1+z_r}\cdot\frac{z_r}{u}\cdot x_r,\label{ProductF-Th-2}\\
		\widetilde{\eta}_r \equiv 
		\widetilde{\eta}(u,z_r,y_r)&:=\frac{1+u}{1+z_r}\cdot y_r,  \label{ProductF-Th-3}
	\end{align}
	and
	\begin{equation}\label{ProductF-Th-4}
		\omega_{A,B}(u):=\left(1+\frac{1}{u}\right)^{A}(1+u)^{B}\frac{1}{u}.
	\end{equation}
\end{theorem}
\begin{proof}
	From \eqref{Def-MLaguerre} and \eqref{ConfluentLauricella}, we obtain
	\begin{align}\label{ProductFormula-MLaguerre-Proof-1}
		&L_{\mathbf{m}}^{(\alpha)}(\mathbf{x}) L_{\mathbf{n}}^{(\beta)}(\mathbf{y})
		\notag\\
		&=\frac{(\alpha+1)_{\langle\mathbf{m}\rangle}}{\mathbf{m}!} 
		\frac{(\beta+1)_{\langle\mathbf{n}\rangle}}{\mathbf{n}!}
		\sum_{\mathbf{i}=\mathbf{0}}^{\infty} 
		\frac{(-m_1)_{i_1}\cdots (-m_k)_{i_k}}{(\alpha+1)_{\langle\mathbf{i}\rangle}}
		\frac{\mathbf{x}^{\mathbf{i}}}{\mathbf{i}!} 
		\sum_{\mathbf{j}=\mathbf{0}}^{\infty} 
		\frac{(-n_1)_{j_1}\cdots (-n_k)_{j_k}}{(\beta+1)_{\langle\mathbf{j}\rangle}}
		\frac{\mathbf{y}^{\mathbf{j}}}{\mathbf{j}!}\notag\\
		&=\frac{(\alpha+1)_{\langle\mathbf{m}\rangle}(\beta+1)_{\langle\mathbf{n}\rangle}}{(\mathbf{m}+\mathbf{n})!}  
		\sum_{\mathbf{i}=\mathbf{0}}^{\infty} 
		\sum_{\mathbf{j}=\mathbf{0}}^{\infty} \frac{\Gamma(\alpha+1)\Gamma(\beta+1)}{\Gamma(\alpha+1+\langle\mathbf{i}\rangle)\Gamma(\beta+1+\langle\mathbf{j}\rangle)}\frac{\mathbf{x}^{\mathbf{i}}}{\mathbf{i}!}\frac{\mathbf{y}^{\mathbf{j}}}{\mathbf{j}!}\frac{\Gamma(\alpha+\beta+1+\langle\mathbf{i}\rangle+\langle\mathbf{j}\rangle)}{(\alpha+\beta+1)_{\langle\mathbf{i}\rangle+\langle\mathbf{j}\rangle}\Gamma(\alpha+\beta+1)}\notag\\
		&\hspace{1cm}\cdot\prod_{r=1}^{k}\frac{(-m_r-n_r)_{i_r+j_r}\Gamma(m_r+n_r-i_r-j_r+1)}{\Gamma(m_r-i_r+1)\Gamma(n_r-j_r+1)}
		\notag\\
		&=\frac{(\alpha+1)_{\langle\mathbf{m}\rangle}(\beta+1)_{\langle\mathbf{n}\rangle}}{(\mathbf{m}+\mathbf{n})!} 
		\frac{\Gamma(\alpha+1)\Gamma(\beta+1)}{\Gamma(\alpha+\beta+1)}
		\sum_{\mathbf{i}=\mathbf{0}}^{\infty} 
		\sum_{\mathbf{j}=\mathbf{0}}^{\infty}
		\frac{(-m_1-n_1)_{i_1+j_1}\cdots (-m_k-n_k)_{i_k+j_k}}{(\alpha+\beta+1)_{\langle\mathbf{i}\rangle+\langle\mathbf{j}\rangle}} \frac{\mathbf{x}^{\mathbf{i}}}{\mathbf{i}!}\frac{\mathbf{y}^{\mathbf{j}}}{\mathbf{j}!}\notag\\
		&\hspace{1cm}\cdot\frac{\Gamma(\alpha+\beta+1+\langle\mathbf{i}\rangle+\langle\mathbf{j}\rangle)}{\Gamma(\alpha+1+\langle\mathbf{i}\rangle)\Gamma(\beta+1+\langle\mathbf{j}\rangle)}
		\prod_{r=1}^{k}
		\binom{m_{r}+n_{r}-i_{r}-j_{r}}{m_r-i_r}.
	\end{align}
	Here, for the sake of technical convenience in the subsequent analysis, we temporarily retain the upper limit of the summation as $\infty$.
	
	By making use of \eqref{AuxiliaryIF-1} and \eqref{AuxiliaryIF-2}, we have
	\begin{align}\label{ProductFormula-MLaguerre-Proof-2}
		&\frac{\Gamma(\alpha+\beta+1+\langle\mathbf{i}\rangle+\langle\mathbf{j}\rangle)}{\Gamma(\alpha+1+\langle\mathbf{i}\rangle)\Gamma(\beta+1+\langle\mathbf{j}\rangle)}
		\prod_{r=1}^{k}\frac{\Gamma(m_{r}+n_{r}-i_{r}-j_{r}+1)}{\Gamma(m_r-i_r+1)\Gamma(n_r-j_r+1)}\notag\\
		&\hspace{0.5cm}=\frac{1}{(2\pi\mathrm{i})^{k+1}}\int_{\mathbb{T}^{k+1}}
		\omega_{\alpha,\beta}(u)
		\prod_{r=1}^{k}\omega_{m_r,n_r}(z_r)
		\notag\\
		&\hspace{3.5cm}\cdot\prod_{r=1}^{k}\left(\frac{(1+u)z_r}{(1+z_r)u}\right)^{i_r}\left(\frac{1+u}{1+z_r}\right)^{j_r}
		\mathrm{d}u
		\mathrm{d}\mathbf{z},
	\end{align}
	where $\Re(\alpha+\beta)>-1$ and $\omega_{A,B}(z)$ is defined by \eqref{ProductF-Th-4}. Substituting \eqref{ProductFormula-MLaguerre-Proof-2} into \eqref{ProductFormula-MLaguerre-Proof-1} gives 
	\begin{align*}
		&\frac{L_{\mathbf{m}}^{(\alpha)}(\mathbf{x})}{\Gamma(\alpha+1+\langle\mathbf{m}\rangle)} \frac{L_{\mathbf{n}}^{(\beta)}(\mathbf{y})}{\Gamma(\beta+1+\langle\mathbf{n}\rangle)}\\
		&\hspace{0.5cm}=\frac{1}{(\mathbf{m}+\mathbf{n})!}
		\frac{1}{(2\pi\mathrm{i})^{k+1}\Gamma(\alpha+\beta+1)}\int_{\mathbb{T}^{k+1}}
		\omega_{\alpha,\beta}(u)
		\prod_{r=1}^{k}\omega_{m_r,n_r}(z_r)\\
		&\hspace{1.5cm}\cdot\sum_{\mathbf{i}=\mathbf{0}}^{\infty} 
		\sum_{\mathbf{j}=\mathbf{0}}^{\infty} 
		\frac{(-m_1-n_1)_{i_1+j_1}\cdots (-m_k-n_k)_{i_k+j_k}}{(\alpha+\beta+1)_{\langle\mathbf{i}\rangle+\langle\mathbf{j}\rangle}\,\mathbf{i}!\,\mathbf{j}!} 
		\prod_{r=1}^{k}
		\widetilde{\xi}_r^{i_r}\widetilde{\eta}_r^{j_r}\mathrm{d}u\mathrm{d}\mathbf{z},
	\end{align*}
	where $\widetilde{\xi}_r$ and $\widetilde{\eta}_r$ are defined by \eqref{ProductF-Th-2} and \eqref{ProductF-Th-4} respectively.

	Finally, in view of the identity \cite[p. 4546, Eq. (5)]{Liu-Lin-Lu-Srivastava-2013b}:
	\[
	\sum_{\mathbf{i},\mathbf{j}=\mathbf{0}}^{\infty}\Omega(\mathbf{i}+\mathbf{j})\frac{\mathbf{x}^{\mathbf{i}}}{\mathbf{i}!}\frac{\mathbf{y}^{\mathbf{j}}}{\mathbf{j}!}
	=\sum_{\mathbf{i}=\mathbf{0}}^{\infty}\Omega(\mathbf{i})\frac{(\mathbf{x}+\mathbf{y})^{\mathbf{i}}}{\mathbf{i}!},
	\]
	where $\Omega(\mathbf{i})$ is a suitable $k$-dimensional sequence, we obtain 
	\begin{align*}
		&\frac{L_{\mathbf{m}}^{(\alpha)}(\mathbf{x})}{\Gamma(\alpha+1+\langle\mathbf{m}\rangle)} \frac{L_{\mathbf{n}}^{(\beta)}(\mathbf{y})}{\Gamma(\beta+1+\langle\mathbf{n}\rangle)}\\
		&\hspace{0.5cm}=\frac{1}{(\mathbf{m}+\mathbf{n})!}
		\frac{1}{(2\pi\mathrm{i})^{k+1}\Gamma(\alpha+\beta+1)}\int_{\mathbb{T}^{k+1}}
		\omega_{\alpha,\beta}(u)
		\prod_{r=1}^{k}\omega_{m_r,n_r}(z_r)\\
		&\hspace{1.5cm}\cdot\sum_{\mathbf{i}=\mathbf{0}}^{\infty} 
		\frac{(-m_1-n_1)_{i_1}\cdots (-m_k-n_k)_{i_k}}{(\alpha+\beta+1)_{\langle\mathbf{i}\rangle}\,\mathbf{i}!} 
		\prod_{r=1}^{k}
		(\widetilde{\xi}_r+\widetilde{\eta}_r)^{i_r}\mathrm{d}u\mathrm{d}\mathbf{z}\\
		&\hspace{0.5cm}=
		\frac{1}{(2\pi\mathrm{i})^{k+1}}\int_{\mathbb{T}^{k+1}}
		\omega_{\alpha,\beta}(u)
		\prod_{r=1}^{k}\omega_{m_r,n_r}(z_r)\\
		&\hspace{1.5cm}\cdot
		\frac{L_{\mathbf{m}+\mathbf{n}}^{(\alpha+\beta)}(\widetilde{\xi}_1+\widetilde{\eta}_1,\cdots,\widetilde{\xi}_k+\widetilde{\eta}_k)}{\Gamma(\alpha+\beta+\langle\mathbf{m}+\mathbf{n}\rangle+1)}\mathrm{d}u\mathrm{d}\mathbf{z}.
	\end{align*}
	This completes the proof of the product formula \eqref{ProductF-Th-1}.
\end{proof}
\begin{remark}
	Letting $z_r=\mathrm{e}^{2\mathrm{i}\varphi_r}$ $(r=1,\cdots,k)$ and  $u=\mathrm{e}^{2\mathrm{i}\theta}$ in Theorem \ref{ProductF-Th}, we obtain 
	\begin{align}
		\widetilde{\xi}_r(\mathrm{e}^{2\mathrm{i}\theta},\mathrm{e}^{2\mathrm{i}\varphi_r},x_r)&=x_r\cdot \mathrm{e}^{\mathrm{i}(\varphi_r-\theta)}\cos\theta\sec\varphi_r,
		\label{ProductF-Th-Remark-1}\\
		\widetilde{\eta}_r(\mathrm{e}^{2\mathrm{i}\theta},\mathrm{e}^{2\mathrm{i}\varphi_r},y_r)&=y_r\cdot 
		\mathrm{e}^{\mathrm{i}(\theta-\varphi_r)}\cos\theta\sec\varphi_r,
		\label{ProductF-Th-Remark-2}
	\end{align}
	and 
	\begin{align}
		&\frac{1}{(2\pi\mathrm{i})^{k+1}}\omega_{\alpha,\beta}(u)
		\prod_{r=1}^{k}\omega_{m_r,n_r}(z_r)\mathrm{d}u\mathrm{d}\mathbf{z}\\
		&\hspace{0.5cm}=\frac{2^{\alpha+\beta+\langle\mathbf{m}+\mathbf{n}\rangle}}{\pi^{k+1}}
		\mathrm{e}^{\mathrm{i}(\beta-\alpha)\theta+\mathrm{i}\langle(\mathbf{n}-\mathbf{m})\circ\bm{\varphi}\rangle}
		\cos^{\alpha+\beta}\theta
		\cos^{m_1+n_1}\varphi_1\cdots \cos^{m_k+n_k}\varphi_k\mathrm{d}\theta\mathrm{d}\bm{\varphi}. 
		\label{ProductF-Th-Remark-3}
	\end{align}
	Substituting \eqref{ProductF-Th-Remark-1}, \eqref{ProductF-Th-Remark-2} and \eqref{ProductF-Th-Remark-3} into \eqref{ProductF-Th-1}, and letting $\mathbf{x}\leftrightarrow\mathbf{y}$, $\mathbf{m}\leftrightarrow\mathbf{n}$ and $\alpha\leftrightarrow\beta$ in the resulting equation, we reproduce the product formula \eqref{ProductFormula-MLaguerre}. 
\end{remark}

At the end of this section, we would like to briefly discuss the possibility of extending Watson's product formula \eqref{WatsonIntegral-1939} to the multivariate case. Watson's original proof \cite{Watson-1938} relies on the classical Hardy-Hille formula \cite[p. 461, Eq. (18.18.27)]{NIST Handbook}
	\begin{equation}\label{HardyHilleF}
		\sum_{n=0}^{\infty}
		\frac{n!L_{n}^{(\alpha)}(x)L_{n}^{(\alpha)}(y)}{\Gamma(\alpha+1+n)}u^n
		=\frac{(xyu)^{-\frac{\alpha}{2}}}{1-u}
		\exp\left(-\frac{u(x+y)}{1-u}\right)
		I_{\alpha}\left(\frac{2\sqrt{xyu}}{1-u}\right).
	\end{equation}
	By comparing \eqref{GF-Erdelyi} with \eqref{HardyHilleF}, we realize that Watson's method \cite{Watson-1939} works only when $x_1=\cdots=x_k=x$ and $y_1=\cdots=y_k=y$. In that case, however, the multivariate Laguerre polynomials degenerate into the univariate Laguerre polynomials. To see this, we return to the generating function \eqref{GF-MLaguerre} and find that
	\begin{align*}
		\sum_{\mathbf{n}=\mathbf{0}}^{\infty}L_{\mathbf{n}}^{(\alpha)}(x,\cdots,x)\mathbf{u}^{\mathbf{n}}
		&=(1-\langle\mathbf{u}\rangle)^{-\alpha-1}\exp\left(-\frac{x\langle\mathbf{u}\rangle}{1-\langle\mathbf{u}\rangle}\right)\\
		&=\sum_{m=0}^{\infty}L_{m}^{(\alpha)}(x)\langle\mathbf{u}\rangle^{m}
		=\sum_{\mathbf{n}=\mathbf{0}}^{\infty}\frac{\langle\mathbf{n}\rangle!}{\mathbf{n}!}L_{\langle\mathbf{n}\rangle}^{(\alpha)}(x)\mathbf{u}^{\mathbf{n}}.
	\end{align*}
	Thus, we have
	\begin{equation}\label{ReductionF}
		L_{\mathbf{n}}^{(\alpha)}(x,\cdots,x)=\frac{\langle\mathbf{n}\rangle!}{\mathbf{n}!}L_{\langle\mathbf{n}\rangle}^{(\alpha)}(x).
	\end{equation}
	So we may not expect Watson's formula to have a nontrivial multivariable generalization.

\section{Oshima's fractional calculus and the Srivastava-Niukkanen formula}\label{Sec-OshimaFC}

Let $E_k$ be the standard simplex in $\mathbb{R}^k$. The Dirichlet measure $\mathrm{d}\mu_{(\bm{\lambda},\alpha)}(\mathbf{s})$, introduced by Carlson \cite[p. 64, Definition 4.4-1]{Carlson-Book-1977} is defined on $E_k$ by
\begin{equation}\label{DirichletMeasure}
	\mathrm{d}\mu_{(\bm{\lambda},\alpha)}(\mathbf{s})
	=\frac{\Gamma(\langle\bm{\lambda}\rangle+\alpha)}{\Gamma(\lambda_1)\cdots\Gamma(\lambda_k)\Gamma(\alpha)}\mathbf{s}^{\bm{\lambda}-1}(1-\langle\mathbf{s}\rangle)^{\alpha-1}\mathrm{d}\mathbf{s},
\end{equation}
where $\lambda_j>0$ $(j=1,\cdots,k)$ and $\alpha>0$.  Srivastava and Niukkanen \cite[p. 249, Eq. (22)]{Srivastava-Niukkanen} established the following multiple integral representation for the multivariate Laguerre polynomials: 
\begin{align}\label{Srivastava-Niukkanen Integral}
	L_{\mathbf{n}}^{(\langle\bm{\lambda}\rangle+\alpha+k)}(\mathbf{x})
	=\frac{(\langle\bm{\lambda}\rangle+\alpha+k+1)_{\langle\mathbf{n}\rangle}}{(\lambda_1+1)_{n_1}\cdots (\lambda_k+1)_{n_k}} 
	\int_{E_k} 
	\prod_{j=1}^{k} 
	L_{n_j}^{(\lambda_j)}(x_j s_j)\mathrm{d}\mu_{(\bm{\lambda}+\mathbf{1},\alpha+1)}(\mathbf{s}),
\end{align}
where $\lambda_j>-1$ $(j=1,\cdots,k)$ and $\alpha>-1$. This formula plays an important role in deriving upper bounds for multivariate Laguerre polynomials \cite{Luo-Raina-2026}. Here we revisit the formula \eqref{Srivastava-Niukkanen Integral} from the viewpoint of Oshima's fractional calculus \cite{Oshima-2024}.

As usual, the Riemann-Liouville fractional integral operator $I_{0+}^{\alpha}$ is defined by
\[
I_{0+}^{\alpha}f(x):=\frac{1}{\Gamma(\alpha)}\int_0^x (x-t)^{\alpha-1}f(t)\mathrm{d}t. 
\]
If we perform the change of variable $t=xs$, then we can write 
\[
I_{0+}^{\alpha}f(x)=x^{\alpha} K_x^\alpha f(x),
\] 
where the integral transformation $K_x^\mu$ is defined by
\[
K_x^\alpha f(x):=\frac{1}{\Gamma(\alpha)}\int_0^1 (1-s)^{\alpha-1}f(xs)\mathrm{d}s.
\]

In 2025, Oshima \cite[p. 554]{Oshima-2024} extended the definition of $K_x^{\alpha}$ to a multivariate function $f(\mathbf{x})$ as follows: 
\[
K_{\mathbf{x}}^{\alpha}f(\mathbf{x}):=\frac{1}{\Gamma(\alpha)}
\int_{E_k}(1-\langle\mathbf{s}\rangle)^{\alpha-1}f(\mathbf{x}\circ\mathbf{s})\mathrm{d}\mathbf{s}.
\]
A simple calculation gives 
\[
K_{\mathbf{x}}^{\alpha}\mathbf{x}^{\bm{\lambda}}=\frac{\Gamma(\lambda_1)\cdots\Gamma(\lambda_k)}{\Gamma(\lambda_1+\cdots+\lambda_k+\alpha)}\mathbf{x}^{\bm{\lambda}}.
\]
Then Oshima defined the transformation \cite[p. 557]{Oshima-2024}
\begin{equation}\label{OshimaOperator}
	K_{\mathbf{x}}^{\alpha,\bm{\lambda}}:=\mathbf{x}^{\mathbf{1}-\bm{\lambda}}K_{\mathbf{x}}^{\alpha}\mathbf{x}^{\bm{\lambda}-\mathbf{1}}
\end{equation}
on the ring $\mathcal{O}_0$ of convergent power series of $\mathbf{x}$, and showed that if 
\[
u(\mathbf{x})=\sum_{\mathbf{m}=\mathbf{0}}^{\infty}c_{\mathbf{m}}\mathbf{x}^{\mathbf{m}}\in\mathcal{O}_0,
\]
then 
\[
K_{\mathbf{x}}^{\alpha,\bm{\lambda}}u(\mathbf{x})
=\frac{\Gamma(\lambda_1)\cdots\Gamma(\lambda_k)}{\Gamma(\langle\bm{\lambda}\rangle+\alpha)}\sum_{\mathbf{m}=\mathbf{0}}^{\infty}\frac{(\lambda_1)_{m_1}\cdots(\lambda_k)_{m_k}}{(\langle\bm{\lambda}\rangle+\alpha)_{\langle\mathbf{m}\rangle}}c_{\mathbf{m}}\mathbf{x}^{\mathbf{m}}.
\]

Next, let us interpret Oshima's operator \eqref{OshimaOperator} from a different angle. Note that the operator \eqref{OshimaOperator} can be explicitly written as 
\begin{align*}
	K_{\mathbf{x}}^{\alpha,\bm{\lambda}} f(\mathbf{x})
	&=\frac{1}{\Gamma(\alpha)}\mathbf{x}^{1-\bm{\lambda}}
	\int_{E_k}
	(\mathbf{1}-\langle\mathbf{s}\rangle)^{\alpha-1}
	(\mathbf{x}\circ\mathbf{s})^{\bm{\lambda}-1}
	f(\mathbf{x}\circ\mathbf{s})\mathrm{d}\mathbf{s}\\
	&=\frac{1}{\Gamma(\alpha)}
	\int_{E_k}
	\mathbf{s}^{\bm{\lambda}-1}
	(\mathbf{1}-\langle\mathbf{s}\rangle)^{\alpha-1}
	f(\mathbf{x}\circ\mathbf{s})\mathrm{d}\mathbf{s}.
\end{align*}
Thus, with the help of \eqref{DirichletMeasure}, we have equivalently 
\begin{equation}\label{OshimaOperator-1}
	K_{\mathbf{x}}^{\alpha,\bm{\lambda}} f(\mathbf{x})
	=\frac{\Gamma(\lambda_1)\cdots\Gamma(\lambda_k)}{\Gamma(\langle\bm{\lambda}\rangle+\alpha)}
	\int_{E_k}f(\mathbf{x}\circ\mathbf{s})\mathrm{d}\mu_{(\bm{\lambda},\alpha)}(\mathbf{s}).
\end{equation}

Letting $\bm{\lambda}\rightarrow\bm{\lambda}+\mathbf{1}$ and $\alpha\rightarrow\alpha+1$ in \eqref{OshimaOperator-1} gives 
\begin{equation}\label{OshimaOperator-2}
	K_{\mathbf{x}}^{\alpha+1,\bm{\lambda}+\mathbf{1}} f(\mathbf{x})
	=\frac{\Gamma(\lambda_1+1)\cdots\Gamma(\lambda_k+1)}{\Gamma(\langle\bm{\lambda}+\mathbf{1}\rangle+\alpha+1)}
	\int_{E_k}f(\mathbf{x}\circ\mathbf{s})\mathrm{d}\mu_{(\bm{\lambda}+\mathbf{1},\alpha+1)}(\mathbf{s}).
\end{equation}
By setting 
\[
f(\mathbf{x}\circ\mathbf{s})=\prod_{j=1}^{k} 
L_{n_j}^{(\lambda_j)}(x_j s_j)
\]
in \eqref{OshimaOperator-2} and then using \eqref{Srivastava-Niukkanen Integral}, we obtain
\begin{align*}
	K_{\mathbf{x}}^{\alpha+1,\bm{\lambda}+\mathbf{1}} \prod_{j=1}^{k} 
	L_{n_j}^{(\lambda_j)}(x_j)
	&=\frac{\Gamma(\lambda_1+1)\cdots\Gamma(\lambda_k+1)}{\Gamma(\langle\bm{\lambda}\rangle+\alpha+k+1)}
	\int_{E_k}\prod_{j=1}^{k} 
	L_{n_j}^{(\lambda_j)}(x_j s_j)\mathrm{d}\mu_{(\bm{\lambda}+\mathbf{1},\alpha+1)}(\mathbf{s})\\
	&=\frac{\Gamma(\lambda_1+1)\cdots\Gamma(\lambda_k+1)}{\Gamma(\langle\bm{\lambda}\rangle+\alpha+k+1)}
	\frac{(\lambda_1+1)_{n_1}\cdots (\lambda_k+1)_{n_k}}{(\langle\bm{\lambda}\rangle+\alpha+k+1)_{\langle\mathbf{n}\rangle}}
	L_{\mathbf{n}}^{(\langle\bm{\lambda}\rangle+\alpha+k)}(\mathbf{x}).
\end{align*}
A slight rearrangement of the above expression yields the following proposition.

\begin{proposition}\label{Prop-FIR-MLaguerre}
	Let $\lambda_j>-1$ $(j=1,\cdots,k)$ and $\alpha>-1$. We have
	\[
	L_{\mathbf{n}}^{(\langle\bm{\lambda}\rangle+\alpha+k)}(\mathbf{x})
	=
	\frac{\Gamma(\langle\bm{\lambda}+\mathbf{n}\rangle+\alpha+k+1)}{\Gamma(\lambda_1+1+n_1)\cdots\Gamma(\lambda_k+1+n_k)}
	K_{\mathbf{x}}^{\alpha+1,\bm{\lambda}+\mathbf{1}} \prod_{j=1}^{k} 
	L_{n_j}^{(\lambda_j)}(x_j).
	\]
\end{proposition}

With the aid of Proposition \ref{Prop-FIR-MLaguerre}, we rather inadvertently obtain from the result of Srivastava and Singhal \cite{Srivastava-Singhal-1972} a new generating function for the multivariate Laguerre polynomials. Recall that (see \cite[p. 1239, Eq. (5)]{Srivastava-Singhal-1972}; see also \cite[p. 9, Eq. (1.13)]{Srivastava-1984} and \cite[p. 468, Eq. (1)]{Srivastava-Manocha-Book-1984})
\begin{align}\label{Srivastava-Singhal-GF}
	&\sum_{\mathbf{n}=\mathbf{0}}^{\infty}\frac{(m+\langle\mathbf{n}\rangle)!}{(\lambda_1+1)_{n_1}\cdots (\lambda_k+1)_{n_l}}
	L_{m+\langle\mathbf{n}\rangle}^{(\beta)}(x)
	\prod_{j=1}^{k} 
	L_{n_j}^{(\lambda_j)}(y_j)
	\mathbf{u}^{\mathbf{n}}=(\beta+1)_m(1-\langle\mathbf{u}\rangle)^{-\beta-m-1}\mathrm{e}^x\notag\\
	&\hspace{2cm}\cdot
	\Psi_2^{(k+1)}\left[\beta+k+1;\bm{\lambda}+\mathbf{1},\beta+1;\frac{u_1 y_1}{\langle\mathbf{u}\rangle-1},\cdots,\frac{u_k y_k}{\langle\mathbf{u}\rangle-1},\frac{x}{\langle\mathbf{u}\rangle-1}\right],
\end{align}
where 
\begin{equation}\label{Def-Psi2}
	\Psi_2^{(k+1)}[a;\mathbf{c},c_{k+1};\mathbf{x},x_{k+1}]
	:=\sum_{\substack{\mathbf{n}=\mathbf{0},\\ n_{k+1}=0}}^{\infty}\frac{(a)_{\langle\mathbf{n}\rangle+n_{k+1}}}{(c_1)_{n_1}\cdots(c_k)_{n_k}(c_{k+1})_{n_{k+1}}}\frac{\mathbf{x}^\mathbf{n}}{\mathbf{n}!}\cdot\frac{x_{k+1}^{n_{k+1}}}{n_{k+1}!},
\end{equation}
is a confluent form of the Lauricella function $F_C^{(k+1)}$ or $F_A^{(k+1)}$ \cite[p. 62, Eq. (13)]{Srivastava-Manocha-Book-1984}. At the very first glance, this formula \eqref{Srivastava-Singhal-GF} appears very complex and is a rather isolated result. However, in what follows, we shall show that from a higher-dimensional perspective, this result is quite a concise result.

Applying $K_{\mathbf{y}}^{\alpha+1,\bm{\lambda}+\mathbf{1}}$ on both sides of \eqref{Srivastava-Singhal-GF}, we obtain
\begin{align*}
	&\sum_{\mathbf{n}=\mathbf{0}}^{\infty}\frac{(m+\langle\mathbf{n}\rangle)!L_{m+\langle\mathbf{n}\rangle}^{(\beta)}(x)
		L_{\mathbf{n}}^{(\langle\bm{\lambda}\rangle+\alpha+k)}(\mathbf{y})}{(\langle\bm{\lambda}\rangle+\alpha+k+1)_{\langle\mathbf{n}\rangle}}
	\mathbf{u}^{\mathbf{n}}\\
	&\hspace{1cm}=\frac{\Gamma(\langle\bm{\lambda}\rangle+\alpha+k+1)}{\Gamma(\lambda_1+1)\cdots\Gamma(\lambda_k+1)}(\beta+1)_m(1-\langle\mathbf{u}\rangle)^{-\beta-m-1}\mathrm{e}^x\\
	&\hspace{1.5cm}\cdot K_{\mathbf{y}}^{\alpha+1,\bm{\lambda}+\mathbf{1}}
	\Psi_2^{(k+1)}\left[\beta+m+1;\bm{\lambda}+\mathbf{1},\beta+1;\frac{u_1 y_1}{\langle\mathbf{u}\rangle-1},\cdots,\frac{u_k y_k}{\langle\mathbf{u}\rangle-1},\frac{x}{\langle\mathbf{u}\rangle-1}\right],
\end{align*}
where
\begin{align*}
	& K_{\mathbf{y}}^{\alpha+1,\bm{\lambda}+\mathbf{1}}
	\Psi_2^{(k+1)}\left[\beta+m+1;\bm{\lambda}+\mathbf{1},\beta+1;\frac{u_1 y_1}{\langle\mathbf{u}\rangle-1},\cdots,\frac{u_k y_k}{\langle\mathbf{u}\rangle-1},\frac{x}{\langle\mathbf{u}\rangle-1}\right]\\
	&\hspace{0.5cm}=\frac{\Gamma(\lambda_1+1)\cdots\Gamma(\lambda_k+1)}{\Gamma(\langle\bm{\lambda}\rangle+k+1+\alpha)}
	\sum_{\substack{\mathbf{n}=\mathbf{0},\\n_{k+1}=0}}^{\infty}
	\frac{(\beta+m+1)_{\langle\mathbf{n}\rangle+n_{k+1}}}{(\langle\bm{\lambda}\rangle+k+\alpha+1)_{\langle\mathbf{n}\rangle}(\beta+1)_{n_{k+1}}}\\
	&\hspace{1.5cm}\cdot \left(\frac{u_1 y_1}{\langle\mathbf{u}\rangle-1}\right)^{n_1}
	\cdots\left(\frac{u_k y_k}{\langle\mathbf{u}\rangle-1}\right)^{n_k}
	\left(\frac{x}{\langle\mathbf{u}\rangle-1}\right)^{n_{k+1}}\\
	&\hspace{0.5cm}=\frac{\Gamma(\lambda_1+1)\cdots\Gamma(\lambda_k+1)}{\Gamma(\langle\bm{\lambda}\rangle+k+1+\alpha)}
	\sum_{n_{k+1}=0}^{\infty}\frac{(\beta+m+1)_{n_{k+1}}}{(\beta+1)_{n_{k+1}} n_{k+1}!}\left(\frac{x}{\langle\mathbf{u}\rangle-1}\right)^{n_{k+1}}\\
	&\hspace{1.5cm}\cdot\sum_{\mathbf{n}=\mathbf{0}}^{\infty}
	\frac{(\beta+m+1+n_{k+1})_{\langle\mathbf{n}\rangle}}{(\langle\bm{\lambda}\rangle+k+\alpha+1)_{\langle\mathbf{n}\rangle}\mathbf{n}!}
	\left(\frac{u_1 y_1}{\langle\mathbf{u}\rangle-1}\right)^{n_1}
	\cdots\left(\frac{u_k y_k}{\langle\mathbf{u}\rangle-1}\right)^{n_k}
	\\
	&\hspace{0.5cm}=\frac{\Gamma(\lambda_1+1)\cdots\Gamma(\lambda_k+1)}{\Gamma(\langle\bm{\lambda}\rangle+k+1+\alpha)}
	\sum_{n_{k+1}=0}^{\infty}\frac{(\beta+m+1)_{n_{k+1}}}{(\beta+1)_{n_{k+1}} n_{k+1}!}\left(\frac{x}{\langle\mathbf{u}\rangle-1}\right)^{n_{k+1}}\\
	&\hspace{1.5cm}\cdot\sum_{\ell=0}^{\infty}
	\frac{(\beta+m+1+n_{k+1})_{\ell}}{(\langle\bm{\lambda}\rangle+k+\alpha+1)_{\ell}\ell!}
	\left(\frac{\langle\mathbf{u}\circ\mathbf{y}\rangle}{\langle\mathbf{u}\rangle-1}\right)^{\ell}
	\\
	&\hspace{0.5cm}=\frac{\Gamma(\lambda_1+1)\cdots\Gamma(\lambda_k+1)}{\Gamma(\langle\bm{\lambda}\rangle+k+1+\alpha)}
	\sum_{\ell=0}^{\infty}\frac{(\beta+m+1)_{\ell}}{(\beta+1)_{\ell} \ell!}\left(\frac{x}{\langle\mathbf{u}\rangle-1}\right)^{\ell}
	{}_{1}F_{1}\left[\begin{matrix}
		\beta+m+1+\ell\\
		\langle\bm{\lambda}\rangle+k+\alpha+1
	\end{matrix};\frac{\langle\mathbf{u}\circ\mathbf{y}\rangle}{\langle\mathbf{u}\rangle-1}\right]
	\\
	&\hspace{0.5cm}=\frac{\Gamma(\lambda_1+1)\cdots\Gamma(\lambda_k+1)}{\Gamma(\langle\bm{\lambda}\rangle+k+1+\alpha)}
	\Psi_2^{(2)}\left[\beta+m+1;\beta+1;\langle\bm{\lambda}\rangle+k+\alpha+1;\frac{x}{\langle\mathbf{u}\rangle-1}, \frac{\langle\mathbf{u}\circ\mathbf{y}\rangle}{\langle\mathbf{u}\rangle-1}\right].
\end{align*}
Now we have
\begin{align}\label{Srivastava-Singhal-GF-1}
	&\sum_{\mathbf{n}=\mathbf{0}}^{\infty}\frac{(m+\langle\mathbf{n}\rangle)!}{(\langle\bm{\lambda}\rangle+\alpha+k+1)_{\langle\mathbf{n}\rangle}}
	L_{m+\langle\mathbf{n}\rangle}^{(\beta)}(x)
	L_{\mathbf{n}}^{(\langle\bm{\lambda}\rangle+\alpha+k)}(\mathbf{y})
	\mathbf{u}^{\mathbf{n}}=(\beta+1)_m(1-\langle\mathbf{u}\rangle)^{-\beta-m-1}\mathrm{e}^x\notag\\
	&\hspace{2.5cm}\cdot \Psi_2^{(2)}\left[\beta+m+1;\beta+1;\langle\bm{\lambda}\rangle+k+\alpha+1;\frac{x}{\langle\mathbf{u}\rangle-1}, \frac{\langle\mathbf{u}\circ\mathbf{y}\rangle}{\langle\mathbf{u}\rangle-1}\right].
\end{align}
Finally, letting $\langle\bm{\lambda}\rangle+\alpha+k\rightarrow\lambda$ in \eqref{Srivastava-Singhal-GF-1} gives
\begin{corollary}
	We have
	\begin{align}\label{Srivastava-Singhal-GF-2}
		&\sum_{\mathbf{n}=\mathbf{0}}^{\infty}\frac{(m+\langle\mathbf{n}\rangle)!}{(\lambda+1)_{\langle\mathbf{n}\rangle}}
		L_{m+\langle\mathbf{n}\rangle}^{(\beta)}(x)
		L_{\mathbf{n}}^{(\lambda)}(\mathbf{y})
		\mathbf{u}^{\mathbf{n}}\notag\\
		&\hspace{1cm}=(\beta+1)_m(1-\langle\mathbf{u}\rangle)^{-\beta-m-1}\mathrm{e}^x \Psi_2^{(2)}\left[\beta+m+1;\beta+1;\lambda+1;\frac{x}{\langle\mathbf{u}\rangle-1}, \frac{\langle\mathbf{u}\circ\mathbf{y}\rangle}{\langle\mathbf{u}\rangle-1}\right].
	\end{align}
\end{corollary}

Note that when $k=1$, \eqref{Srivastava-Singhal-GF-2} corresponds to the following well-known result (\cite[p. 1240, Eq. (11)]{Srivastava-Singhal-1972}; see also \cite[p. 135, Eq. (20)]{Srivastava-Manocha-Book-1984})
\begin{align}\label{Srivastava-Singhal-GF-3}
	&\sum_{n=0}^{\infty}\frac{(m+n)!}{(\lambda+1)_{n}}
	L_{m+n}^{(\beta)}(x)
	L_{n}^{(\lambda)}(y)
	u^{n}\notag\\
	&\hspace{1cm}=(\beta+1)_m(1-u)^{-\beta-m-1}\mathrm{e}^x \Psi_2^{(2)}\left[\beta+m+1;\beta+1;\lambda+1;\frac{x}{u-1}, \frac{uy}{u-1}\right].
\end{align}

It is worth pointing out that the generating function \eqref{Srivastava-Singhal-GF-2} suggests the existence of the following more general result. 

\begin{theorem}\label{Th-ThakurtaType}
	If there exists a multiple generating relation of the form
	\begin{equation}\label{Th-ThakurtaType-1}
		G(x,\mathbf{u})=\sum_{\mathbf{n}=\mathbf{0}}^{\infty}a_{\mathbf{n}}L_{\langle\mathbf{n}\rangle+m}^{(\alpha)}(x)\mathbf{u}^{\mathbf{n}},
	\end{equation}
	then
	\begin{equation}\label{Th-ThakurtaType-2}
		(1-\langle\mathbf{u}\rangle)^{-\alpha-m-1}
		\exp\left(-\frac{\langle\mathbf{u}\rangle x}{1-\langle\mathbf{u}\rangle}\right)
		G\left(\frac{x}{1-\langle\mathbf{u}\rangle},\frac{\mathbf{u}\circ\mathbf{z}}{1-\langle\mathbf{u}\rangle}\right)
		=\sum_{\mathbf{n}=\mathbf{0}}^{\infty} A_{\mathbf{n}}(\mathbf{z}) L_{\langle\mathbf{n}\rangle+m}^{(\alpha)}(x)\mathbf{u}^{\mathbf{n}},
	\end{equation}
	where
	\[
	A_{\mathbf{n}}(\mathbf{z})=\sum_{\mathbf{j}=\mathbf{0}}^{\mathbf{n}}a_{\mathbf{j}}\frac{(\langle\mathbf{n}\rangle+m+1)!}{(\langle\mathbf{j}\rangle+m+1)!}\frac{\mathbf{z}^{\mathbf{j}}}{(\mathbf{n}-\mathbf{j})!}.
	\]
\end{theorem}
\begin{proof}
	Let (\cite[p. 532]{Thakurta-1987})
		\[
		\mathcal{R}:=xy\frac{\partial}{\partial x}+y^2\frac{\partial}{\partial y}+(-x+m+1)y, 
		\]
		then
		\begin{equation}\label{Th-ThakurtaType-Proof-1}
			\mathcal{R}\left\{y^{\alpha+\langle\mathbf{n}\rangle}L_{\langle\mathbf{n}\rangle+m}(x)\right\}=(\langle\mathbf{n}\rangle+m+1)y^{\alpha+\langle\mathbf{n}\rangle+1}L_{\langle\mathbf{n}\rangle+m+1}^{(\alpha)}(x).
		\end{equation}
		Replacing $\mathbf{u}$ by $y(\mathbf{u}\circ\mathbf{z})$ in \eqref{Th-ThakurtaType-1} and multiplying both sides of the resulting identity by $y^{\alpha}$, we obtain 
		\begin{equation}\label{Th-ThakurtaType-Proof-2}
			y^{\alpha}G(x,y(\mathbf{u}\circ\mathbf{z}))=\sum_{\mathbf{n}=\mathbf{0}}^{\infty}a_{\mathbf{n}}(\mathbf{u}\circ\mathbf{z})^{\mathbf{n}} y^{\langle\mathbf{n}\rangle+\alpha}L_{\langle\mathbf{n}\rangle+m}^{(\alpha)}(x).
		\end{equation}
		Operating both sides of \eqref{Th-ThakurtaType-Proof-2} by $\mathrm{e}^{\langle\mathbf{u}\rangle\mathcal{R}}$, we have
		\begin{equation}\label{Th-ThakurtaType-Proof-3}
			\mathrm{e}^{\langle\mathbf{u}\rangle\mathcal{R}}\left\{y^{\alpha}G(x,y(\mathbf{u}\circ\mathbf{z}))\right\}
			=\mathrm{e}^{u_1\mathcal{R}}
			\cdots 
			\mathrm{e}^{u_k\mathcal{R}}
			\left\{\sum_{\mathbf{n}=\mathbf{0}}^{\infty}a_{\mathbf{n}}(\mathbf{u}\circ\mathbf{z})^{\mathbf{n}} y^{\langle\mathbf{n}\rangle+\alpha}L_{\langle\mathbf{n}\rangle+m}^{(\alpha)}(x)\right\}.
		\end{equation}
		The right-hand side of \eqref{Th-ThakurtaType-Proof-3} may be evaluated as 
		\begin{align}\label{Th-ThakurtaType-Proof-4}
			&\sum_{\mathbf{n}=\mathbf{0}}^{\infty}
			\sum_{\mathbf{j}=\mathbf{0}}^{\infty}
			a_{\mathbf{n}}(\mathbf{u}\circ\mathbf{z})^{\mathbf{n}}
			\frac{\mathbf{u}^{\mathbf{j}}}{\mathbf{j}!}
			\mathcal{R}^{\langle\mathbf{j}\rangle}\left\{ y^{\langle\mathbf{n}\rangle+\alpha}L_{\langle\mathbf{n}\rangle+m}^{(\alpha)}(x)\right\}\notag\\
			&=y^{\alpha}\sum_{\mathbf{n}=\mathbf{0}}^{\infty}
			\sum_{\mathbf{j}=\mathbf{0}}^{\infty}
			a_{\mathbf{n}}\mathbf{u}^{\mathbf{n}+\mathbf{j}}\mathbf{z}^{\mathbf{n}}
			\frac{(\langle\mathbf{n}\rangle+\langle\mathbf{j}\rangle+m+1)!}{\mathbf{j}!(\langle\mathbf{n}\rangle+m+1)!} 
			y^{\langle\mathbf{n}\rangle+\langle\mathbf{j}\rangle}L_{\langle\mathbf{n}\rangle+\langle\mathbf{j}\rangle+m}^{(\alpha)}(x)\notag\\
			&=y^{\alpha}\sum_{\mathbf{n}=\mathbf{0}}^{\infty}
			y^{\langle\mathbf{n}\rangle}L_{\langle\mathbf{n}\rangle+m}^{(\alpha)}(x)\mathbf{u}^{\mathbf{n}}
			\left(\sum_{\mathbf{j}=\mathbf{0}}^{\mathbf{n}}
			a_{\mathbf{n}-\mathbf{j}}\mathbf{z}^{\mathbf{n}-\mathbf{j}}
			\frac{(\langle\mathbf{n}\rangle+m+1)!}{\mathbf{j}!(\langle\mathbf{n}\rangle-\mathbf{j}+m+1)!}\right) \notag\\
			&=y^{\alpha}\sum_{\mathbf{n}=\mathbf{0}}^{\infty}
			A_{\mathbf{n}}(\mathbf{z})L_{\langle\mathbf{n}\rangle+m}^{(\alpha)}(x)(y\mathbf{u})^{\mathbf{n}}, 
		\end{align}
		while the left-hand side of \eqref{Th-ThakurtaType-Proof-3} may be evaluated as
		\begin{equation}\label{Th-ThakurtaType-Proof-5}
			\mathrm{e}^{\langle\mathbf{u}\rangle\mathcal{R}}\left\{y^{\alpha}G(x,y(\mathbf{u}\circ\mathbf{z}))\right\}
			=(1-\langle\mathbf{u}y\rangle)^{-\alpha-m-1}
			\exp\left(-\frac{\langle\mathbf{u}\rangle xy}{1-\langle\mathbf{u}\rangle y}\right) y^\alpha 
			G\left(\frac{x}{1-\langle\mathbf{u}\rangle y},\frac{y(\mathbf{u}\circ\mathbf{z})}{1-\langle\mathbf{u}\rangle y}\right).
		\end{equation}
		The bilateral generating function given by \eqref{Th-ThakurtaType-2} is obtained by 
		equating \eqref{Th-ThakurtaType-Proof-4} and \eqref{Th-ThakurtaType-Proof-5} and then setting $y=1$.
\end{proof}
\begin{remark}
		Theorem \ref{Th-ThakurtaType} can be viewed as a multivariate generalization of Thakurta's result \cite[p. 531, Theorem 1]{Thakurta-1987}. In addition, the technique used in the proof has a deep group-theoretic background and has been widely applied to find generating functions for various sequences of univariate polynomials (see, for example, \cite{Chongdar-Mukherjee-1988}, \cite{Chongdar-Chongdar-Mohanta-2024} and \cite{Sharma-Chongdar-1991}). However, its use in the multivariate setting is still rare. The present theorem offers a very illustrative example.
\end{remark}

Note that, by letting 
	\[
	a_{\mathbf{j}}=\frac{(\langle\mathbf{j}\rangle+m+1)!}{(\lambda+1)_{\langle\mathbf{j}\rangle}}\frac{(-1)^{\langle\mathbf{j}\rangle}}{\mathbf{j}!},
	\]
	we have 
	\[
	A_{\mathbf{n}}(\mathbf{z})
	=\frac{(\langle\mathbf{n}\rangle+m+1)!}{\mathbf{n}!}\sum_{\mathbf{j}=\mathbf{0}}^{\mathbf{n}}\frac{(-n_1)_{j_1}\cdots(-n_k)_{j_k}}{(\lambda+1)_{\langle\mathbf{j}\rangle}}\frac{\mathbf{z}^{\mathbf{j}}}{\mathbf{j}!}
	=\frac{(\langle\mathbf{n}\rangle+m+1)!}{(\lambda+1)_{\langle\mathbf{n}\rangle}}L_{\mathbf{n}}^{(\alpha)}(\mathbf{z}).
	\]
	Thus, more generating functions of type \eqref{Srivastava-Singhal-GF-2} can be derived by specializing Theorem \ref{Th-ThakurtaType}. We leave this to interested readers.

\section{Concluding remarks}\label{Sec-ConcludingRemarks}

In the present work we have studied different types of generating functions, product formulas and fractional integral representation for Erd\'{e}lyi's multivariate Laguerre polynomials $L_{\mathbf{n}}^{(\alpha)}(\mathbf{x})$. In this concluding section, we would aim at providing an interesting evaluation of a ``non-traditional'' generating function involving the \emph{main diagonal sequence} and leave an open problem for future research.

Recall that if 
\[
F(\mathbf{z})=\sum_{\textbf{n=0}}^{\infty}f_{\mathbf{n}}\mathbf{z}^{\mathbf{n}}
\]
is a (formal or convergent) multivariate power series, the \emph{main diagonal} of $F(\mathbf{z})$ is the univariate power series
\[
(\Delta F)(u)=\sum_{n=0}^{\infty}f_{n,\cdots,n}u^n.
\]
For example, we have (\cite[p. 102, Example 3.4]{Melczer-Book-2021})
\[
\Delta \left(\frac{1}{1-x-y}\right)
=\Delta\left(\sum_{i,j=0}^{\infty}\binom{i+j}{i}x^i y^j\right)
=\sum_{n=0}^{\infty}\binom{2n}{n} u^n
=\frac{1}{\sqrt{1+4u}}.
\]
An extensive theory of this type of generating function has already been developed and their applications can be found in many combinatorial problems (see \cite{Pemantle-Wilson-Book-2013}). Here, we are going to evaluate
\[
\Delta\left(\mathrm{e}^{-\langle\mathbf{u}\circ \mathbf{x}\rangle}(1+\langle\mathbf{u}\rangle)^{-\beta}\right)
=\sum_{n=0}^{\infty}L_{n,\cdots,n}^{(-\beta-kn)}(\mathbf{x})u^n.
\]

From the integral representation \eqref{ContourIR-MLaguerre}, we have
\begin{align}
	\sum_{n=0}^{\infty}L_{n,\cdots,n}^{(-\beta-kn)}(\mathbf{x})u^n
	&=\frac{\Gamma(1-\beta)}{2\pi\mathrm{i}}\sum_{n=0}^{\infty}
	\frac{u^n}{(n!)^k}\int_{-\infty}^{(0+)}\mathrm{e}^s s^{\beta-1}(s-x_1)^{n}
	\cdots 
	(s-x_k)^{n}\mathrm{d}s\notag\\
	&=\frac{\Gamma(1-\beta)}{2\pi\mathrm{i}}\int_{-\infty}^{(0+)}\mathrm{e}^s s^{\beta-1}
	F_k\big(u(s-x_1)
	\cdots 
	(s-x_k)\big)\mathrm{d}s\notag\\
	&=:G_k(\mathbf{x},u),\label{IntegralRofGF-1}
\end{align}
where $F_\gamma(z)$ is the \emph{Le Roy function} defined by (\cite{Le Roy-1899}; refer also to \cite{Le Roy-1900} )
\begin{equation}\label{Def-LeRoy}
	F_\gamma (z):=\sum_{n=0}^{\infty}\frac{z^n}{(n!)^{\gamma}}~~~(\gamma\in\mathbb{C}).       \end{equation}
It is well-known that $F_k(z)$ $(k\in\mathbb{Z}_{\geq1})$ is an entire function of order $1/k$ and type $1$, and is also \emph{holonomic} (or \emph{$D$-finite}). In recent years, the Le Roy function and its generalizations have attracted attention of many researchers (see, for example, \cite{Garrappa-Rogosin-Mainardi-2017, Gerhold-2012} and \cite{Rogosin-Dubatovskaya-2023}).

When $\gamma=1$ in \eqref{Def-LeRoy}, it follows that $F_1(z)=\mathrm{e}^z$, and thus from \eqref{IntegralRofGF-1} we have                                
\[
G_1(x,u)
=\mathrm{e}^{-ux}\frac{\Gamma(1-\beta)}{2\pi\mathrm{i}}\int_{-\infty}^{(0+)}\mathrm{e}^{(1+u)s} s^{\beta-1}\mathrm{d}s
=(1+u)^{-\beta}\mathrm{e}^{-ux},
\]
which agrees with the classial result \eqref{GF-Laguerre-2}. In general, however, the contour integral cannot be reduced to an expression involving only elementary functions or known special functions. Under the condition $\Re(\beta)>0$, we can follow usual method of handling such type of contour integrals to obtain
\begin{equation}\label{IntegralRofGF-2}
	G_k(\mathbf{x},u)=\frac{1}{\Gamma(\beta)}\int_{0}^{\infty}\mathrm{e}^{-s} s^{\beta-1}
	F_{k}\big((-1)^k u(s+x_1)
	\cdots 
	(s+x_k)\big)\mathrm{d}s.
\end{equation}

By writing $\displaystyle \|\mathbf{x}\|:=\max_{1\leq j\leq k}\{|x_j|\}$, we have
\[
\int_{0}^{\infty}\mathrm{e}^{-s} s^{\Re(\beta)-1}
\left|F_{k}\big((-1)^k u(s+x_1)
\cdots 
(s+x_k)\big)\right|\mathrm{d}s
\leq \int_{0}^{\infty}\mathrm{e}^{-s} s^{\Re(\beta)-1}
F_k\big(|u|(s+\|\mathbf{x}\|)^k\big)\mathrm{d}s.
\]
For real $z$, recall that (\cite[p. 398, Theorem 1]{Gerhold-2012})
\[
F_k(z)\sim C_k z^{(1-k)/(2k)}\mathrm{e}^{kz^{1/k}}, \quad z\rightarrow+\infty,
\]
where $C_k:=(2\pi)^{(1-k)/2}k^{-1/2}$ $(k\in\mathbb{Z}_{\geq1})$. We have
then
\[
F_k\big(|u|(s+\|\mathbf{x}\|)^k\big)\sim C_k 
|u|^{(1-k)/(2k)}
\mathrm{e}^{k|u|^{1/k}\|\mathbf{x}\|}
\cdot (s+\|\mathbf{x}\|)^{(1-k)/2}
\mathrm{e}^{k|u|^{1/k}s}, \quad s\rightarrow+\infty. 
\]
Therefore, the convergence of the integral in \eqref{IntegralRofGF-2} is guaranteed by the condition $|u|<1/k^k$.

Finally, we believe that the investigation of the asymptotic behaviour of $L_{n,\cdots,n}^{(-\beta-kn)}(\mathbf{x})$ as $n\rightarrow+\infty$ by means of the generating function given by \eqref{IntegralRofGF-1} or \eqref{IntegralRofGF-2} would be an interesting problem, and we leave this as a worthwhile proposal for future research.

\section*{Funding}

No funding was received for conducting this study.

\section*{Conflicts of interest}

The authors declare that there is no conflict of interest.

\section*{Data Availability}

This manuscript has no associated data.

\section*{Acknowledgements}

We are grateful to Peng-Cheng Hang and Shu-Ying Shen for their careful reading of the initial manuscript and providing valuable suggestions.


\end{document}